\documentclass[a4paper,10pt]{article}
\usepackage{amsmath,amsfonts,amssymb,amsthm}
\usepackage{graphicx,subfigure,booktabs,multirow}
\usepackage{mathrsfs,enumerate,color,bm}
\usepackage{algorithm,algpseudocode}
\usepackage[textwidth=14cm]{geometry}
\usepackage{color}

\usepackage{hyperref}

\newtheorem{theorem}{Theorem}[section]
\newtheorem{lemma}[theorem]{Lemma}
\newtheorem{proposition}[theorem]{Proposition}

\newtheorem{definition}{Definition}[section]

\newcommand{\R}{\mathbb{R}}

\newcommand{\Qmn}{\mathbb{H}^{n_1\times n_2}}

\newcommand{\rank}{\operatorname{rank}}

\newcommand{\re}{\operatorname{Re}}
\newcommand{\diag}{\operatorname{diag}}

\newcommand{\spn}{\operatorname{span}}

\graphicspath{{fig/}}

\title{Low Rank Pure Quaternion Approximation for Pure Quaternion Matrices}
\author{
Guangjing Song\thanks{School of Mathematics and Information Sciences, Weifang University, Weifang 261061, P.R. China. Email: {\tt sgjshu@163.com}}
\and
Weiyang Ding\thanks{
Institute of Science and Technology for Brain-Inspired Intelligence, Fudan University,
Shanghai 200433, P.R. China. Research supported in part by NSFC Grant 11801479 and HKRGC GRF 12301619. Email: {\tt
weiyang.ding@gmail.com}}
\and
Michael K. Ng\thanks{The Corresponding Author. Department of Mathematics, The University of Hong Kong, Pokfulam, Hong Kong. Research supported in part by the HKRGC GRF 
12200317, 12300218, 12300519 and 17201020. E-mail: {\tt mng@maths.hku.hk}}
}

\begin{document}
\maketitle

\begin{abstract}
Quaternion matrices are employed successfully in many color image processing applications.
In particular, a pure 
quaternion matrix can be used to represent red, green and blue channels of color images.
A low-rank approximation for a 
pure quaternion matrix can be obtained by using the quaternion singular value decomposition.
However, this approximation is not optimal in the sense that the resulting low-rank approximation matrix may not be pure quaternion, i.e.,
the low-rank matrix contains real component which is not useful for the representation of a color image.
The main contribution of this paper is to find an optimal rank-$r$ 
pure quaternion matrix approximation for a 
pure quaternion matrix (a color image).
Our idea is to use
a projection on a low-rank quaternion matrix manifold and a projection on a quaternion matrix with zero real component, and develop
an alternating projections algorithm to find such optimal low-rank 
pure quaternion matrix approximation.
The convergence of the projection algorithm
can be established by showing that the low-rank quaternion matrix manifold and the zero real component quaternion matrix manifold
has a non-trivial intersection point.
Numerical examples on synthetic 
pure quaternion matrices and color images
are presented to illustrate the projection algorithm can find optimal low-rank 
pure quaternion approximation
for pure quaternion matrices or color images.
\end{abstract}

\vspace{3mm}

\section{Introduction}

The RGB color model is one of the most commonly applied additive color models.
Each pixel using the RGB color model consists of three channels, i.e., red (R), green (G), and blue(B), which can be encoded on the three imaginary parts of a quaternion.
The quaternion representation for color images has been proposed and widely employed in the literature \cite{chang2003quaternion,ell2006hypercomplex,han2013color,le2004singular,pei2001efficient,sangwine1996fourier,subakan2011quaternion}.
It is a main advantage of the quaternion approach that color images can be studied and processed holistically as a vector field \cite{ell2006hypercomplex,sangwine1996fourier,subakan2011quaternion}.
Tools and methods for gray-scale image processing are extended to the color image processing field via the quaternion algebra theory and computation, such as the matrix completion \cite{jia2019robust}, the Fourier transform \cite{pei2001efficient,sangwine1996fourier}, the wavelet transform \cite{fletcher2017development}, the principal component analysis \cite{zeng2016color}, and the dictionary learning algorithms \cite{barthelemy2015color,xu2015vector,yu2013quaternion}.

In this paper, we also employ the quaternion matrix representation for an RGB color image and target at the optimal rank-$r$ approximation that should also represent an RGB color image.
Note that the representation quaternion matrix should be pure quaternion, 
i.e., admitting a zero real part, since we assign the red, green, and blue channels to the three imaginary parts respectively.
Without the pure quaternion restriction, a low-rank approximation of any quaternion matrix can be obtained by the quaternion singular value decomposition (QSVD) \cite{zhang1997quaternions}, i.e., the quaternion counterpart for the Eckart-Young-Mirsky theorem.
Neverthelss, this low-rank approximation may not be optimal in the sense that the approximation matrix is not necessarily a pure quaternion matrix.
That is, the approximation matrix contains real component information which is useless for a color image.
We take a 2-by-2 pure quaternion matrix for illustration. Denote
\begin{eqnarray}\label{ex1}
{\bf A}  = \left(
  \begin{array}{cc}
    1 & 0  \\
    0 & 0  \\
  \end{array}
\right) {\bf i} +\left(
                   \begin{array}{cc}
                     0 & 1  \\
                     0 & 1  \\
                   \end{array}
                 \right){\bf j}+\left(
                                  \begin{array}{cc}
                                    0 & 0  \\
                                    1 & 0  \\
                                  \end{array}
                                \right){\bf k},
\end{eqnarray}
where ${\bf i}$, ${\bf j}$ and ${\bf k}$ are three imaginary units in the quaternion algebra.
 By applying the QSVD on ${\bf A},$  we can calculate an optimal rank-$1$ approximation as
$$
\tilde{{\bf A}} = \left(
              \begin{array}{cc}
                0 & -0.35 \\
                0 & 0.35  \\
              \end{array}
            \right)+\left(
                      \begin{array}{cc}
                        0.85 & 0  \\
                        0.35 & 0  \\
                      \end{array}
                    \right){\bf i}
                    +\left(
                                     \begin{array}{cc}
                                       0 & -0.85 \\
                                       0 & -0.85  \\
                                     \end{array}
                                   \right){\bf j}+\left(
                                                    \begin{array}{cc}
                                                      0.35 & 0  \\
                                                      0.85 & 0  \\
                                                    \end{array}
                                                  \right){\bf k}.
$$
Obviously, the approximation matrix $\tilde{{\bf A}}$ contains real components.
 In this case,  the pure quaternion part of $\tilde{{\bf A}}$
$$
\tilde{{\bf A}}_{p}=\left(
                      \begin{array}{cc}
                        0.85 & 0  \\
                        0.35 & 0  \\
                      \end{array}
                    \right){\bf i}
                    +\left(
                                     \begin{array}{cc}
                                       0 & -0.85 \\
                                       0 & -0.85  \\
                                     \end{array}
                                   \right){\bf j}+\left(
                                                    \begin{array}{cc}
                                                      0.35 & 0  \\
                                                      0.85 & 0  \\
                                                    \end{array}
                                                  \right){\bf k}
$$
is usually chosen as the  
pure quaternion approximation of ${\bf A}$. However,
$\textrm{rank} \left(\tilde{{\bf A}}_{p}\right)=2,$
which is saying that $\tilde{{\bf A}}_{p}$ is not an optimal rank-$1$ 
pure quaternion approximation of ${\bf A}$.

In color image processing, we are interested in finding an optimal fixed-rank 
pure quaternion approximation of a color image.
Mathematically, it can be formulated as the following optimization problem
\begin{equation}\label{pmain}
\begin{array}{cl}
  \min & \| {\bf A}- {\bf X}\|_{\textrm{F}}^{2}, \\
  {\rm s.t.} & \rank({\bf X})=r, \\
             & \re({\bf X})=0,
\end{array}
\end{equation}
where ${\bf A}$ is a given 
pure quaternion matrix and $\re({\bf X})$ denotes the real part of ${\bf X}$.
In the literature, there are several algorithms \cite{jia2013new,jia2017color,le2004singular,li2014fast,li2016real}
for computing the eigenvalues as well as the singular values of quaternion matrices.
To our best knowledge, this is the first attempt to study the low-rank 
pure quaternion approximation problem.

\subsection{The Contribution}
 Alternating projection method finds a point in the intersection of two closed convex sets by iteratively projecting a point first onto one set and then the other.  It has a long history which can be traced back to John Von Neumann \cite{von1950functional}, where the alternating projections between two closed subspaces of a Hilbert space is guaranteed to globally converge to a intersection point of the two subspaces, if they intersect non-trivially.   Alternating projection method  has been used in a wide range of some classical mathematics problems as well as engineering applications because of it is a gradient-free method (without requiring tuning the step size) and usually has fast speed of convergence. See, e.g., convex feasibility problem \cite{bauschke1996projection}, signal processing \cite{combettes1993signal},  finance \cite{higham2002computing}, machine learning \cite{widrow1987adaptive},  an so on (\cite{adamyan1968general,grigoriadis1994application,hamaker1978angles,kayalar1988error,lee1997conformal,levi1983signal,tanabe1971projection,youla1978generalized} and references therein).

By contrast,  alternating projection method on nonconvex sets are remained rather undeveloped. Cadzow \cite{cadzow1988signal} showed the convergence of an alternating projections scheme by the Zangwill's global convergence theorem. Lewis and Malich \cite{lewis2008alternating} further showed that alternating projection method converges locally at a linear rate when two manifolds intersect transversally.
Later, Fredrik and Marcus \cite{andersson2013alternating} generalized  the ``transversally'' intersecting condition as the ``nontangentially'' intersecting condition.


In this paper, we consider the alternating projection method on two manifolds: one is a fixed-rank quaternion matrix manifold and the other is a zero real component quaternion matrix manifold.
The convergence results for  the alternating projections algorithm on these two manifolds are derived. Furthermore, we propose an initialization strategy to make the alternating projection method more practical and reliable.
Numerical examples on synthetic pure quaternion matrices and color images are presented to illustrate the performances of the proposed algorithm.

The rest of this paper is organized as follows.
In Section \ref{sec:pre}, we summarize some notations used throughout this paper.
The preliminaries for quaternion matrix and  manifolds theory are presented.
We also study the intersection of the low-rank quaternion matrix manifold and the zero real component quaternion matrix manifold.
In Section \ref{sec:main}, an alternating projections algorithm is given and proved to linearly converge to a good approximation.
We further propose an initialization strategy in Section \ref{sec:initia} to make the alternating projection method
more practical and reliable.
In Section \ref{sec:exper}, we conduct some numerical experiments to demonstrate our theoretical results.
Some concluding remarks are given in Section \ref{sec:con}.

\section{Mathematical Preliminaries}
\label{sec:pre}


\subsection{Quaternion and Quaternion Matrix}

A quaternion number consists of one real part and three imaginary parts given by
$$
{\bf q} = q_r + q_i {\bf i} + q_j {\bf j} + q_k {\bf k}
$$
where $q_r, q_i, q_j, q_k \in \mathbb{R}$ and ${\bf i}, {\bf j}, {\bf k}$ are three imaginary units.
Throughout this paper, any boldface symbol indicates a quaternion number, vector, or matrix.
We use $\mathbb{H}$ to denote the quaternion algebra.
The quaternion ${\bf q}$ is called a pure quaternion if it has a zero real part, i.e., $q_r = 0$.
The conjugate and modulus of ${\bf q}$ are respectively defined by
$$
{\bf q}^* = q_r - q_i {\bf i} - q_j {\bf j} - q_k {\bf k} \quad {\rm and} \quad
| {\bf q} | = \sqrt{ q_r^2 + q_i^2 + q_j^2 + q_k^2}.
$$
Quaternions form a division algebra when equipped with the componentwise addition, the componentwise scalar multiplication over $\mathbb{R}$, and the Hamilton product given by
\begin{itemize}
  \item ${\bf i} \cdot 1 = 1 \cdot {\bf i} = {\bf i}$, ${\bf j} \cdot 1 = 1 \cdot {\bf j} = {\bf j}$, ${\bf k} \cdot 1 = 1 \cdot {\bf k} = {\bf k}$;
  \item ${\bf i}^2 = {\bf j}^2 = {\bf k}^2 = -1$, ${\bf i} {\bf j} = - {\bf j} {\bf i} = {\bf k}$, ${\bf j} {\bf k} = - {\bf k} {\bf j} = {\bf i}$, ${\bf i} {\bf k} = - {\bf k} {\bf i} = {\bf j}$.
\end{itemize}
Note that the quaternion multiplication is noncommutative, i.e., ${\bf p}{\bf q}$ may not equal ${\bf q}{\bf p}$  for all ${\bf p},{\bf q}\in \mathbb{H}$.

A quaternion matrix is represented by ${\bf A}=A_{0}+A_{1}{\bf i} +A_{2}{\bf j} +A_{3}{\bf k} \in \mathbb{H}^{m\times n}$ and the real part of ${\bf A}$ is denoted by $\re({\bf A})=A_{0}$.
The identity quaternion matrix ${\bf I}$ is the same as the classical identity matrix.
The inverse ${\bf B}$ of a quaternion matrix ${\bf A}$ exists if ${\bf A} {\bf B} = {\bf B} {\bf A} = {\bf I}$.
A quaternion matrix is unitary if ${\bf A}^* {\bf A} = {\bf A} {\bf A}^* = {\bf I}$, where ${\bf A}^*$ is the conjugate transpose of ${\bf A}$.
An alternative approach to handle an $ m\times n $ quaternion matrix is to consider the subset of the ring ${\mathbb R}^{4m\times 4n}$:
\begin{align*}
\Gamma:=\left\{X=
            \left(\begin{array}{cccc}
             A_{0} & -A_{1} & A_{2} & -A_{3} \\
             A_{1} & A_{0} &-A_{3} & A_{2} \\
             -A_{2} & A_{3}& A_{0} & -A_{1} \\
             A_{3} & -A_{2} & A_{1} & A_{0} \\
           \end{array}\right)
\in \mathbb{R}^{4m\times 4n}:\ A_0,A_1,A_2,A_3\in \mathbb{R}^{m\times n}\right\}.
\end{align*}
Inheriting the operations on ${\mathbb R}^{4m\times 4n}$, $\Gamma$ actually forms a subring.
Define a mapping $\phi$ as
\begin{align} \label{realq}
  \phi:~{\bf A}=A_{0}+A_{1}{\bf i} +A_{2}{\bf j} +A_{3}{\bf k}\in \mathbb{H}^{m\times n} \mapsto \hat{\bf A}=\left(
           \begin{array}{cccc}
             A_{0} & -A_{1} & A_{2} & -A_{3} \\
             A_{1} & A_{0} &-A_{3} & A_{2} \\
             -A_{2} & A_{3}& A_{0} & -A_{1} \\
             A_{3} & -A_{2} & A_{1} & A_{0} \\
           \end{array}
         \right)\in \Gamma.
\end{align}
 Then  $\phi$ is a bijection and preserves the operations, which guarantees that the quaternion matrix set $\mathbb{H}^{m\times n}$ and the real matrix set $\Gamma$ are
essentially the same (Remark 2.1 in \cite{zhang1997quaternions}). In addition, $\hat{{\bf A}}$ in \eqref{realq} is called the real representation of ${\bf A}$.

Due to the non-commutative nature of quaternion multiplication, the linear independence of a set of vectors  over $\mathbb{H}^{n}$ can be defined as right  and left linear independence, respectively.
${\bf v}_{1},..,{\bf v}_{n} \in \mathbb{H}^{n}$ are said to be right (left) linear independent if there dose not exist nonzero  quaternions ${\bf a}_{1},...,{\bf a}_{n}$  such that
 \begin{equation*}
 {\bf v}_{1}{\bf a}_{1}+\cdots+{\bf v}_{n}{\bf a}_{n}=0~({\bf a}_{1} {\bf v}_{1}+\cdots+{\bf a}_{n}{\bf v}_{n}=0).
 \end{equation*}
Based on the definition of right independence of quaternion vectors, we can introduce the rank of a quaternion matrix.
\begin{definition}[\cite{zhang1997quaternions}]\label{d:2.1} The maximum
number of right linearly independent columns of a quaternion matrix
${\bf A} \in \Qmn$ is called the rank of ${\bf A}$.
\end{definition}

Nevertheless, the rank could be different from the maximum number of left linearly independent columns or right linearly independent rows since the quaternion multiplication is noncommutative.
 Moreover, the singular value decomposition of a quaternion matrix reveals the rank of the quaternion matrix which can be given as follows.
\begin{theorem} [QSVD \cite{zhang1997quaternions}] \label{t:qsvd}
Let ${\bf A} \in \mathbb{H}^{n_1\times n_2}$ with $\rank({\bf A})=r$. Then there exist two unitary quaternion matrices ${\bf U}
= [ {\bf u}_1, {\bf u}_2, \cdots, {\bf u}_{n_1} ] \in\mathbb{H}^{n_1\times n_1}$
and ${\bf V} = [ {\bf v}_1, {\bf v}_2, \cdots {\bf v}_{n_2} ] \in\mathbb{H}^{n_2\times n_2}$ such that
\begin{equation}\label{e:qsvd}
{\bf A} = {\bf U}  {\bm \Sigma} {\bf V}^*,
\end{equation}
where $ {\bm \Sigma}=\diag(\sigma_1,\cdots,\sigma_r,0,\cdots,0) \in
\mathbb{R}^{n_1\times n_2}$ and  $\sigma_1 \geq \sigma_2 \geq \cdots \geq \sigma_r$ are
positive singular values of ${\bf A}$.
\end{theorem}

Compared with the real and complex cases,  the singular values of a quaternion matrix are still real numbers, while the two unitary matrices are quaternion matrices.
Similar to the complex matrix case, an optimal rank-$r$ approximation of ${\bf A}$ can be given as $\sum_{i=1}^{r} \sigma_{i}{\bf u}_{i}{\bf v}_{i}^{*}$.

\subsection{Manifolds}
 A smooth manifold is a pair $(\mathcal{M},\varphi)$, where $\mathcal{M}$ is a topological manifold and $\varphi$ is a smooth structure on $\mathcal{M}.$
 Let $\gamma:\R\rightarrow \mathcal{M},t\rightarrow \gamma(t)$ be a smooth  curve in $\mathcal{M}$ with $\gamma(0)=x,$ and $\Upsilon(x)$ be the set of all real-valued functions $f:\mathcal{M}\rightarrow \R$ which are smooth around $x\in \mathcal{M}.$ Then the mapping
$$v_{x}:\Upsilon(x)\rightarrow \R, f\rightarrow v_{x}f:=\frac{df(\gamma(t))}{dt}\mid_{t=0}$$
is called the tangent vector of $x$ to the curve $\gamma$
 at $t = 0.$ The set of all derivations of $\mathcal{M}$ at $x$ is a vector
space called the tangent space to $\mathcal{M}$ at $x,$ and is denoted by ${\cal T}_{\cal M}(x).$

Some well known matrix manifolds  in the literature are the orthogonal Stiefel manifold,
the Grassmann manifold and the fixed rank matrix manifold which can be constructed from $\R^{m\times n}$ by taking embedded or quotient operations.  In order to better understand manifolds with its related definitions (e.g., charts, atlases and tangent spaces) and some optimization algorithms on matrix manifolds, we refer to \cite{absil2009optimization,lee2013smooth} and
the references therein. Here, we focus on the fixed rank matrix manifold whose dimension and tangent space are given as follows.
\begin{lemma}[Proposition 2.1 in \cite{vandereycken2013low}]\label{lem1} Let ${\cal M}$ be the set of $m \times n$ real matrices with a fixed rank $r$.
Then  $\mathcal{M}$ is an embedded manifold of $\mathbb{R}^{m\times n}$ with  dimension $(m+n-r)r.$  Suppose that the skinny SVD of $X\in \cal M$ is given by $X=U\Sigma V^{T},$ with $U\in \R^{m\times r}$ and $V\in \R^{n\times r}$ being two column unitary marices. Its tangent space
${\cal T}_{\cal M}(X)$ at $X$  is given by
\begin{align*}
{\cal T}_{\cal M}(X)
=\left\{[U,U_{\bot}]\left(
                                \begin{array}{cc}
                                  \mathbb{R}^{r\times r} & \mathbb{R}^{r\times (n-r)} \\
                                  \mathbb{R}^{(m-r)\times r} & 0^{(m-r)\times (n-r)} \\
                                \end{array}
                              \right)[V,V_{\bot}]^{T}\right\},
\end{align*}
where $U_{\bot}$ and $V_{\bot}$ are the unitary complements of $U$ and $V$, respectively.
\end{lemma}

In this section, our aim is to show the intersection of the fixed rank quaternion matrix set and the zero real component quaternion matrix set is a manifold, which provides a theoretical guarantee of the local linear convergence for the alternating projection method proposed in the next section.
  The angle between two manifolds plays an important role in checking the alternating projection method can be applied or not, then we need to introduce the following definitions first.
\begin{definition}[Definition 3.1 in \cite{andersson2013alternating}] \label{de1}
Suppose that $\mathcal{M}_{1}$ and $\mathcal{M}_{2}$ are two manifolds, then
given $A\in \mathcal{M}_{1}\cap \mathcal{M}_{2}$, set
\begin{equation*}
F_{j}^{ \xi}=\{B_{j}\in \mathcal{M}_{j}\backslash A, \|B_{j}-A\|\leq \xi ~\text{and}~ B_{j}-A~\bot~ T_{\mathcal{M}_{1}\cap \mathcal{M}_{2}}(A)\},~ j=1,2.
\end{equation*}
If $F_{j}^{ \xi}\neq 0, ~j=1,2,$ for all $\xi >0,$ we define the angle $\alpha(A)$ of $\mathcal{M}_{1}$ and $\mathcal{M}_{2}$ at $A$ as
$$\alpha(A)=cos^{-1}(\sigma(A)),$$
where \begin{equation*}
\sigma(A)=\lim_{\xi\rightarrow 0}\sup_{B_{1}\in F^{\xi}_{1},B_{2}\in F^{\xi}_{2}}\left\{\frac{\left<B_{1}-A,B_{2}-A\right>}{\|B_{1}-A\|\|B_{2}-A\|}\right\}.
\end{equation*}
\end{definition}
\begin{definition}[Definition 3.3 in \cite{andersson2013alternating}]
Suppose that $\mathcal{M}_{1}$ and $\mathcal{M}_{2}$ are two manifolds.
Points $A\in \mathcal{M}_{1}\cap \mathcal{M}_{2}$ where the angle is defined will be called nontrivial intersection points. For such points, we say that $A$ is tangential if $\alpha(A)=0$ and non-tangential if $\alpha(A)>0.$
\end{definition}

For quaternion matrices, we can also show the fixed-rank ones form a manifold. The proof can be found in Appendix.
\begin{lemma}\label{qm1}
Denote $\mathcal{Q}:=\left\{{\bf E}=E_{0}+E_{1}{\bf i} +E_{2}{\bf j} +E_{3}{\bf k}\in {\mathbb H}^{m\times n},~\rank({\bf E})=r\right\}$.
Then  $\mathcal{Q}$ is an  embedded submanifold of $\mathbb{R}^{4m\times 4n}$ with dimension $4(m+n-r)r.$
\end{lemma}

By the real representation of a quaternion matrix given in (\ref{realq}), we can get  $\mathcal{Q}$ is isomorphic to the set of structured real matrices:
\begin{align}\label{v1}
&\mathcal{M}_{r}:=\left\{ X=
           \begin{pmatrix}
             A_{0} & -A_{1} & A_{2} & -A_{3} \\
             A_{1} & A_{0} &-A_{3} & A_{2} \\
             -A_{2} & A_{3}& A_{0} & -A_{1} \\
             A_{3} & -A_{2} & A_{1} & A_{0} \\
           \end{pmatrix}
\in \mathbb{R}^{4m\times 4n}, \ \rank(X)= 4r\right\},
\end{align}
and the set of pure quaternion matrices is isomorphic to the set of structured real matrices:
\begin{align}
\mathcal{M}_{*}:=\left\{X=\left(
           \begin{array}{cccc}
            0 & -A_{1} & A_{2} & -A_{3} \\
             A_{1} & 0  &-A_{3} & A_{2} \\
             -A_{2} & A_{3}&0  & -A_{1} \\
             A_{3} & -A_{2} & A_{1} & 0  \\
           \end{array}
         \right)
\in \mathbb{R}^{4m\times 4n}\right\},\label{v2}
\end{align}
respectively. Obviously, $\mathcal{M}_{*}$ is a linear subspace, thus also a manifold.
After that  the main task of this section can be rewritten as proving
\begin{align}\label{v3}
\mathcal{M}_{r*}:=\mathcal{M}_{r}\cap \mathcal{M}_{*}=\left\{X=\left(
           \begin{array}{cccc}
            0 & -A_{1} & A_{2} & -A_{3} \\
             A_{1} & 0  &-A_{3} & A_{2} \\
             -A_{2} & A_{3}&0  & -A_{1} \\
             A_{3} & -A_{2} & A_{1} & 0  \\
           \end{array}
         \right)
\in \mathbb{R}^{4m\times 4n}, \ \rank(X)=4r\right\}
\end{align}
 is a manifold. Moreover, $\mathcal{M}_{r}$ and $\mathcal{M}_{*}$  can be seen as the vanishing of different polynomials sets over $\mathbb{R}^{n}$,  which motivated us to apply algebraic  geometry methods to study the above problem. Before moving on, we need to introduce the following definitions and results which are needed in the sequel.
\begin{theorem}[Theorem 2.1 in  \cite{whiteneyvarieties}]\label{whi}
Given a real algebraic variety $\mathcal{V}\in \mathbb{R}^{n}$, we can write $\mathcal{V}=\bigcup_{j=0}^{m} \mathcal{V}_{j}$ where each $\mathcal{V}_{j}$ is either void or a $\mathbb{C}^{(\infty)}$-manifold of
dimension $j$. Moreover, each $\mathcal{V}_{j}$ contains at most a finite number of connected components.
\end{theorem}
 Theorem \ref{whi} shows us that the main part of a variety is a manifold.  For a given real algebraic variety $\mathcal{V}\in \mathbb{R}^{n}$, if we identity $\mathbb{R}^{n}$ as a subset of $\mathbb{C}^{n}$ and denote $ \mathbb{I}_{\mathbb{R}}(\mathcal{V}) $ as the set of real polynomials that vanish on $\mathcal{V}$, then $\mathcal{V}$ has a related complex variety given by its Zariski closure
\begin{equation*}\label{eqn1}
\mathcal{V}_{Zar} = \{ z \in \mathbb{C}^{n} : p(z)=0, \ \forall p \in \mathbb{I}_{\mathbb{R}}(\mathcal{V}) \},
\end{equation*}
which is defined as the subset in $\mathbb{C}^{n}$ of common zeros to all polynomials that vanish on $\mathcal{V}$.
 A given real algebraic variety $\mathcal{V}$ is called irreducible if there does not exist any non-trivial decompositions of the form $\mathcal{V}=\mathcal{V}_{1}\bigcup \mathcal{V}_{2}$, where $\mathcal{V}_{1}$ and $\mathcal{V}_{2}$ are real algebraic varieties. 
A point $z \in \mathcal{V}$ is non-singular if it is non-singular in the sense of algebraic geometry as an element of $\mathcal{V}_{Zar}.$ 
Denote $\nabla$ as the gradient operator and set ${\cal N}_{\cal V}(z) = \{ \nabla p(z) : p \in \mathbb{I}_{\mathbb{R}}(\mathcal{V}) \}$.
The set of non-singular points in ${\cal V}$ is denoted as ${\cal V}^{ns}$.
The following results provide some criteria to check whether a point is a non-singular point or not.

\begin{lemma}[Proposition 6.3 in \cite{andersson2013alternating}]\label{q11}
 Let $\mathcal{V}\in \mathbb{R}^{n}$ be a irreducible real algebraic variety of dimension $m$.
 Then $\dim {\cal N}_{\mathcal{V}}(z)\leq n-m$ for all $z\in \mathcal{V}$ and $z$ is
 non-singular if and only if $\dim {\cal N}_{\mathcal{V}}(z)=n-m$.
\end{lemma}
\begin{lemma}[Proposition 6.4 in \cite{andersson2013alternating}]\label{new1}
Let $\mathcal{V}$ be an irreducible real algebraic variety of dimension $m$. Then the decomposition $\mathcal{V}=\bigcup_{j=1}^{m}\mathcal{M}_{j}$ in Theorem \ref{whi} can be chosen as $\mathcal{V}^{ns}=\mathcal{M}_{m}.$
\end{lemma}
\begin{lemma}[Proposition 6.5 in \cite{andersson2013alternating}]\label{lm1}
Suppose that $\mathcal{V}_{1}$ and $\mathcal{V}_{2}$ are irreducible real algebraic varieties and that $\mathcal{V}=\mathcal{V}_{1}\cap \mathcal{V}_{2}$ is irreducible and strictly smaller that both $\mathcal{V}_{1}$ and $\mathcal{V}_{2}.$ Then each point in $\mathcal{V}^{ns}_{1}\cap \mathcal{V}^{ns}_{2}\cap \mathcal{V}^{ns}$ is a non-trivial intersection point.
\end{lemma}
\begin{lemma}[Theorem 6.6 in \cite{andersson2013alternating}]\label{lm2}
Suppose that $\mathcal{V}_{1}$ and $\mathcal{V}_{2}$ are irreducible real algebraic varieties and that $\mathcal{V}=\mathcal{V}_{1}\cap \mathcal{V}_{2}$ is irreducible and strictly smaller that both $\mathcal{V}_{1}$ and $\mathcal{V}_{2}.$ Let the dimension of $\mathcal{V}$ is $m.$ If $\mathcal{V}^{ns,nt}\neq 0,$ then $\mathcal{V}\setminus \mathcal{V}^{ns,nt}$ is a real algebraic variety of dimension strictly less than $m$. A sufficient condition for this to happen is that there exist a point $A\in \mathcal{V}_{1}^{ns}\cap \mathcal{V}_{2}^{ns}$ such that
\begin{equation*}
\dim(T_{\mathcal{V}_{1}^{ns}}(A)+T_{\mathcal{V}_{2}^{ns}}(A))\leq m.
\end{equation*}
\end{lemma}
 In  practice, we need to check a given variety is irreducible or not,  thus the following results are needed.
\begin{definition}[Definition 6.7 in \cite{andersson2013alternating}]\label{newdf}
  Suppose we are given a number $j\in \mathbb{N}$ and an index set $I$ such that for each $i\in I$, there exist an open connected $\Omega_{i}\subseteq \mathbb{R}^{j}$ and a real analytic map $\phi_{i}:\Omega_{i}\rightarrow \mathcal{V}.$  Then $\mathcal{V}$ is said to be covered with analytic patches, if for each $A\in \mathcal{V},$ there exists an $i\in I$ and a radius $r_{A}$ such that
$$\mathcal{V}_{rn}\cap Ball_{\mathbb{R}^{n}}(A,r_{A})=Im \phi_{i}\cap  Ball_{\mathbb{R}^{n}}(A,r_{A}).$$
\end{definition}

\begin{lemma}[Proposition 6.8 in \cite{andersson2013alternating}]\label{q1}
Let $\mathcal{V}$ be a real algebraic variety. If $\mathcal{V}$ is connected and can be covered with analytic patches, then $\mathcal{V}$ is irreducible.
\end{lemma}
The following lemma show us a method to compute the dimension of a variety.
\begin{lemma}[Proposition 6.9 in \cite{andersson2013alternating}]\label{q2}
Under the assumption of Lemma \ref{q1}, suppose in addition that an open subset of $\mathcal{V}$ is the image of a bijective real analytic map
defined on  a subset of $\mathbb{R}^{d}.$ Then $\mathcal{V}$ has dimension $d$.
\end{lemma}

 For an arbitrary quaternion matrix ${\bf A}=A_{0}+A_{1}{\bf i} +A_{2}{\bf j} +A_{3}{\bf k}\in \mathbb{H}^{m\times n},$ denote
\begin{align}\label{v5}
\mathcal{V}_{sr}:=\left\{X=
           \begin{pmatrix}
             A_{0} & -A_{1} & A_{2} & -A_{3} \\
             A_{1} & A_{0} &-A_{3} & A_{2} \\
             -A_{2} & A_{3}& A_{0} & -A_{1} \\
             A_{3} & -A_{2} & A_{1} & A_{0} \\
           \end{pmatrix}
\in \mathbb{R}^{4m\times 4n}, \ \rank(X)\leq 4r\right\}.
\end{align}
Then  we can derive the following results.

\begin{theorem}\label{q3}
Let $\mathcal{M}_{*}$ and $\mathcal{V}_{sr}$ be  given as in \eqref{v2} and \eqref{v5}, respectively. Then $\mathcal{M}_{*}$ is an affine subspaces of dimension
$3mn,$ $\mathcal{V}_{sr}$ is an irreducible real algebraic variety of dimension $4(m+n)r-4r^2$, and
\begin{align}\label{v4}
\mathcal{V}:=\mathcal{V}_{sr}\cap \mathcal{M}_{*}=\left\{X=
           \begin{pmatrix}
            0 & -A_{1} & A_{2} & -A_{3} \\
             A_{1} & 0 &-A_{3} & A_{2} \\
             -A_{2} & A_{3}& 0 & -A_{1} \\
             A_{3} & -A_{2} & A_{1} &0 \\
           \end{pmatrix}
\in \mathbb{R}^{4m\times 4n}, \ \rank(X)\leq 4r\right\}
\end{align} is an irreducible algebraic variety of dimension $3(m+n)r-3r^2$.
\end{theorem}

\begin{proof}
First, denote
\begin{align*}
\Gamma:=\left\{X=
            \left(\begin{array}{cccc}
             A_{0} & -A_{1} & A_{2} & -A_{3} \\
             A_{1} & A_{0} &-A_{3} & A_{2} \\
             -A_{2} & A_{3}& A_{0} & -A_{1} \\
             A_{3} & -A_{2} & A_{1} & A_{0} \\
           \end{array}\right)
\in \mathbb{R}^{4m\times 4n},~A_{i},i=0,1,2,3\in \mathbb{R}^{m\times n}\right\}.
\end{align*} 
Obviously, $\Gamma$ is a linear space of dimension $4mn$.
The set $\mathcal{M}_{*}$ is obtained by adding the constraint $A_{0}=0$ to $\Gamma$.
Thus, it is an affine space with dimension $3mn$.

Second, we will show that $\mathcal{V}_{sr}$ is   an irreducible real algebraic variety with dimension $4(m+n)r-4r^2$.
It is well known that a matrix  has rank $r$ if and only if there exists at least a non-zero $r\times r$ nonsingular minor which is a matrix obtained by deleting $n-r$ rows and columns and all the $(r+1)\times (r+1)$ minors are zeros.
Then if a matrix in $ \Gamma$ has rank $4r$ then there exists at least a non-zero $4r\times 4r$ invertible minor and all $(4r+1)\times (4r+1)$ minors are zero.
Hence, $\mathcal{V}_{sr}$ is the variety induced by the determinants of these minors.
By the quaternion singular value decomposition given in \cite{zhang1997quaternions}, any $A\in \Gamma$ with $\rank(A)\leq 4r$ can be factorized into
\begin{align}\label{qj1}
{ \scriptsize
A=USV^{T}=\left(
           \begin{array}{cccc}
             U_{0} & -U_{1} & U_{2} & -U_{3} \\
             U_{1} & U_{0} &-U_{3} & U_{2} \\
             -U_{2} & U_{3}& U_{0} & -U_{1} \\
             U_{3} & -U_{2} & U_{1} & U_{0} \\
           \end{array}
         \right)\left(
                   \begin{array}{cccc}
                     \Sigma & 0 & 0 & 0 \\
                     0 &  \Sigma & 0 & 0 \\
                     0 & 0 &  \Sigma & 0 \\
                     0 & 0 & 0 &  \Sigma \\
                   \end{array}
                 \right)\left(
           \begin{array}{cccc}
             V_{0} & -V_{1} & V_{2} & -V_{3} \\
             V_{1} & V_{0} &-V_{3} & V_{2} \\
             -V_{2} & V_{3}& V_{0} & -V_{1} \\
             V_{3} & -V_{2} & V_{1} & V_{0} \\
           \end{array}
         \right)^T,}
\end{align}
with $U_{i}\in \mathbb{R}^{m\times r}$, $V_{i}\in \mathbb{R}^{n\times r}$ ($i=0,1,2,3$) and $\Sigma \in \mathbb{R}^{r\times r}$ being a diagonal matrix.
In the other hand, if a  matrix $A$ can be expressed as \eqref{qj1}, then $\rank(A)\leq 4r$, i.e, $A\in \mathcal{V}_{sr}$.  We see that $\mathcal{V}_{sr}$ is connected and can be covered with one real polynomial. Then $\mathcal{V}_{sr}$ is irreducible by Lemma  \ref{q1}.

Next, choose a subset of $\mathcal{V}_{sr}$ with the diagonal elements of $\Sigma$ being nonzero and different with each other.  Then the freedom of the column unitary matrices $U\in\mathbb{R}^{4m\times 4r}$ and $V\in\mathbb{R}^{4n\times 4r}$ in \eqref{qj1} are  $4mr-\frac{r(4r+1)}{2}$ and $4nr-\frac{r(4r+1)}{2}$, respectively. And the freedom of the diagonal  matrix $S\in\mathbb{R}^{4r\times 4r}$ is $r$.
 Thus, the subsets of such matrices can be identified with $\mathbb{R}^{4mr-\frac{r(4r+1)}{2}}$, $\mathbb{R}^{4nr-\frac{r(4r+1)}{2}}$ and $\mathbb{R}^{r}$, respectively.
Denote the inverses of the identification by
\begin{align}
\iota_{1} : \mathbb{R}^{4mr-\frac{r(4r+1)}{2}}\rightarrow \mathbb{R}^{4m\times 4r};~ \iota_{2} : \mathbb{R}^{r}\rightarrow \mathbb{R}^{4r\times 4r};~\iota_{3} : \mathbb{R}^{4mr-\frac{r(4r+1)}{2}}\rightarrow \mathbb{R}^{4n\times 4r};
 \end{align}
and denote $\Omega\subset  \mathbb{R}^{4(m+n)r-4r^2}$ as the open set corresponding to those matrices with $\Sigma$ possessing different diagonal elements.
Define $\phi: \Omega\rightarrow \mathcal{V}_{sr}$ by
\begin{equation}\label{qj2}
\phi(y_{1},y_{2},y_{3})= {\Large \iota_{1}(y_{1})\cdot \iota_{2}(y_{2})\cdot(\iota_{3}(y_{3}))^{T}}.
\end{equation}
It is easy to see that $\phi$ is a polynomial and moreover a bijective correspondence with an open set $\Omega.$
Thus it follows Lemma \ref{q2} that the dimension of $\mathcal{V}_{sr}$ is $4(m+n)r-4r^2.$

Third, we turn our attention to $\mathcal{V}=\mathcal{V}_{sr}\cap \mathcal{M}_{*}$.
Note that $\mathcal{V}$  is obtained by adding the algebraic equations
\begin{equation}\label{eq4}
U_{0}\Sigma V_{0}^T-U_{1}\Sigma V _{1}^T-U_{2}\Sigma V_{2}^T-U_{3}\Sigma V_{3}^T=0
\end{equation} to those entries of matrices in defining $\mathcal{V}_{sr},$  then it is also a  real algebraic variety.

Then, we will apply Lemma \ref{q1}-\ref{q2} to show $\mathcal{V}$ is an irreducible real algebraic variety with dimension $3(m+n)r-3r^2$.

Let ${U}\in \mathbb{R}^{4m\times 4r}$, ${V}\in \mathbb{R}^{4n\times 4r}$ and  ${S}\in \mathbb{R}^{4r\times 4r}$ be defined as \eqref{qj1}.
We set all the elements of $U_{0}$ as  undetermined variables and other values are fixed. Then the $m\times n$ linear equations in \eqref{eq4} relate to the undetermined variables $(U_{0})_{i,j},i=1,...,m, j=1,...,r$ may have $0,1$ or infinite solutions (the number of solutions was decided by the property of $V,\Sigma,$ and the other variables of $U$). Suppose the remaining values are chosen such that the equations have a unique solution relate to every undermined variable $(U_{0})_{i,j},i=1,...,m, j=1,...,r,$ respectively.  Denote the corresponding matrix by $\hat{U}$, after these $(U_{0})_{i,j},i=1,...,m, j=1,...,r$ are fixed. Then,  a real analytic mapping $\theta$ from $\hat{U}$,  $S$ and  ${V}$  to $\mathcal{V}$ can be constructed as follows:
\begin{align} \label{ineq1}
&\theta((U_{1})_{1,1},...,(U_{1})_{m,r},...,(U_{3})_{1,1},...,(U_{3})_{m,r},\sigma_{1},...,\sigma_{r},(V_{0})_{1,1},...,(V_{3})_{n,r}) \nonumber \\
 =&\left(\begin{array}{cccc}
             U_{0} & -U_{1} & U_{2} & -U_{3} \\
             U_{1} & U_{0} &-U_{3} & U_{2} \\
             -U_{2} & U_{3}& U_{0} & -U_{1} \\
             U_{3} & -U_{2} & U_{1} & U_{0} \\
           \end{array}
         \right)\left(
                   \begin{array}{cccc}
                     \Sigma & 0 & 0 & 0 \\
                     0 &  \Sigma & 0 & 0 \\
                     0 & 0 &  \Sigma & 0 \\
                     0 & 0 & 0 &  \Sigma \\
                   \end{array}
                 \right)\left(
           \begin{array}{cccc}
             V_{0} & -V_{1} & V_{2} & -V_{3} \\
             V_{1} & V_{0} &-V_{3} & V_{2} \\
             -V_{2} & V_{3}& V_{0} & -V_{1} \\
             V_{3} & -V_{2} & V_{1} & V_{0} \\
           \end{array}
         \right)^T \nonumber\\
= &\left(
     \begin{array}{cccc}
       W_{0} & -W_{1} & W_{2} & -W_{3} \\
       W_{1} &W_{0}  &- W_{3} & W_{2} \\
       -W_{2} & W_{3} & W_{0} & -W_{1} \\
       W_{3} & -W_{2} & W_{1} & W_{0} \\
     \end{array}
   \right),
\end{align}
with
\begin{align*}
 W_{0}=U_{0}\Sigma V_{0}^T-U_{1}\Sigma V _{1}^T-U_{2}\Sigma V_{2}^T-U_{3}\Sigma V_{3}^T,
 W_{1}=-U_{0}\Sigma V_{1}^T-U_{1}\Sigma V _{0}^T+U_{2}\Sigma V_{3}^T-U_{3}\Sigma V_{2}^T,\\
 W_{2}=U_{0}\Sigma V_{2}^T+U_{1}\Sigma V _{3}^T-U_{2}\Sigma V_{0}^T-U_{3}\Sigma V_{1}^T,
 W_{3}=-U_{0}\Sigma V_{3}^T-U_{1}\Sigma V _{2}^T-U_{2}\Sigma V_{1}^T-U_{3}\Sigma V_{0}^T.
\end{align*}
Note that the entries of $W_{0}$  in (\ref{ineq1}) can be zeros when the variables in $\hat{U}$,  $S$ and  ${V}$ are chosen as above which can guarantee  equations in \eqref{eq4} are satisfied.
It is saying that $\mathcal{V}$ is the image of $\theta$.
Let $\Gamma$ be a particular connected component of $(\hat{U}, \Sigma, V)$.
We establish a function $\psi$ with  $\Gamma$ as follows:
\begin{equation}\label{tr1}
\psi_{\Gamma}(y) = U(y) \Sigma(y) V(y)^T, ~~y \in \Gamma.
\end{equation}
Denote $\mathbb{I}$ as the set of all possible $\pi$ and $\Gamma$.
It can be found that for each matrix in $\mathcal{V}$ is in the image of at least one $\psi_{\Gamma}$ where
$ \Gamma \in \mathbb{I}$.
Then by Definition \ref{newdf},
$\mathcal{V}$ can be covered by $\{ \psi_{\Gamma} \}_{\Gamma \in \mathbb{I}}$.

Furthermore, in order to show $\mathcal{V}$ is irreducible we need to show $\mathcal{V}$ is connected. It is sufficient to prove $\mathcal{V}$ is path connected, i.e., for any two matrices $A, B\in \mathcal{V}$,
there exist a continuous map $f$ from the unit interval $[0,1]$ to $\mathcal{V}$ such that $f(0) = A$ and $f(1) = B$.
Without loss of generality, we show that
for an arbitrary $A \in \mathcal{V}_{rn}$, it is connected with \begin{equation*}
X=\left(
  \begin{array}{cccc}
    \hat{\mathbf{1}} & \hat{\mathbf{1}} & \cdots & \hat{\mathbf{1}} \\
    \hat{\mathbf{1}} & \hat{\mathbf{1}} & \cdots & \hat{\mathbf{1}} \\
    \vdots & \vdots & \ddots & \vdots \\
    \hat{\mathbf{1}} &  \hat{\mathbf{1}}& \cdots & \hat{\mathbf{1}} \\
  \end{array}
\right)
~\text{with}~
\hat{\mathbf{1}}=\left(
  \begin{array}{cccc}
    0 & -1 & 1 & -1 \\
    1 & 0 & -1 & 1 \\
    -1& 1 & 0 & -1 \\
    1& -1& 1 & 0 \\
  \end{array}
\right)
\end{equation*} instead.
Suppose that $A, B\in \mathcal{V}$ are arbitrary and path connected with the $X$ matrix, respectively. Thus, there are two continuous  maps $f$ and $g$  which are from the unit interval $[0,1]$ to $\mathcal{V}$  with $f(0) = A$, $f(1) = X$,
$g(0)=X$ and $g(1)=B$. Setting $\tau(x)=(1-x)f(x)+xg(x)$, it is easy to see that $\tau(x)$ is continuous and satisfying $\tau(0)=f(0)=A$ and $\tau(1)=g(1)=B$. Then $\mathcal{V}$ is  path connected.
Let $A$ be fixed.
We assume that the diagonal elements of $\Sigma$ are nonnegative and decreasingly ordered and $\Sigma_{ii}=1, i=1,2,3,4$.
Pick $\sigma$ such that $\sigma (i)=\sigma(j)=k$ for all $i,j$ and choose $\Omega$ such that the representation is in the form \eqref{qj2}.
If the second diagonal value in $\Sigma$ is negative, then we continuously change it to the positivity inside $\Omega$.
Then the values of $y$ corresponding to columns 1 through 8 of $U$ can be continuously moved until all elements of the columns 1 through 4 as $[\mathbf{1},\mathbf{1},\cdots,\mathbf{1}]^T$ with
  \begin{align*}\label{mat1}
  {\bf{1}}=\left(
                                                                         \begin{array}{cccc}
                                                                           1 & -1 & 1 &-1 \\
                                                                           1 & 1 & -1 & 1 \\
                                                                           -1 & 1 & 1 & -1 \\
                                                                           1 & -1 & 1 & 1 \\
                                                                         \end{array}
                                                                       \right).
  \end{align*}
At this point, all values of $U$ except the first fourth columns vanish, increasing the first value of each row whenever necessary to stay inside $\Omega$.
We can move $y$ so that the columns 1  though 4  of $V$ can be written as $[\mathbf{a},\mathbf{a},\cdots,\mathbf{a}]^T$ which satisfies  $\mathbf{1}\cdot \mathbf{a}=\hat{\mathbf{1}}$. Thus,  the matrix $\hat{\mathbf{1}}$ can be arrived which is saying that $\mathcal{V}$ is connected.

In the end, we need to determine the dimension of $\mathcal{V}$.  Consider again the map introduced earlier as \eqref{qj2}, with the difference that the diagonal blocks are zeros. In order to guarantee \eqref{eq4} is satisfied, $(m+n)r-r^2$ additional constraints are added on these variables.
It is naturally to define a real analytic map on the open subset $\Xi$ of $\mathbb{R}^{3(m+n)r-3r^2}$.
By \eqref{qj1},  there exist three matrices $U,V$ and $S$ such that $A=US V^{T}$. The sets of $U\in \mathbb{R}^{4m\times 4r},~S =\diag\{\Sigma,\Sigma,\Sigma,\Sigma\}\in \mathbb{R}^{4r\times 4r}$ and  $V\in \mathbb{R}^{4n\times4 r}$ contain $3mr-\frac{r(3r+1)}{2}, r,$ and $ 3nr-\frac{r(3r+1)}{2}$ independent variables, respectively.
Therefore, $U, S$ and $V$ identify the set of matrices with $\mathbb{R}^{3mr-\frac{r(3r+1)}{2}}, $ $\mathbb{R}^{r}$ and $\mathbb{R}^{3nr-\frac{r(3r+1)}{2}},$ respectively.
Then we can identify the set of such matrices with $\mathbb{R}^{3(m+n)r-3r^2}.$
Denote the inverse of the identification as
\begin{align*}
\iota_{1} : \mathbb{R}^{3mr-\frac{r(3r+1)}{2}}\rightarrow \mathbb{R}^{4m\times 4r}; ~\iota_{2} : \mathbb{R}^{r}\rightarrow \mathbb{R}^{4r\times 4r};~\iota_{3} : \mathbb{R}^{3mr-\frac{r(3r+1)}{2}}\rightarrow \mathbb{R}^{4n\times 4r};
\end{align*}
and $\Omega\subset  \mathbb{R}^{3(m+n)r-3r^2}$ as the open set corresponding to matrices with this structure.
$\Psi$ is a bijection with an open subset of $\mathcal{V}$.
Hence,  by  Lemma \ref{q2} we can derive the dimension of $\mathcal{V}$ is $3(m+n)r-3r^2.$
\end{proof}
 Moreover, we can get the following results.
\begin{theorem}\label{th3}
 Suppose that $\mathcal{M}_{r},$  $\mathcal{M}_{*}$ and $\mathcal{V}_{sr}$ are defined as \eqref{v1}, \eqref{v2} and \eqref{v5}, then $\mathcal{V}_{sr}^{ns}=\mathcal{M}_{r}$ and
\begin{equation}\label{neweq1}
\dim(\spn(T_{\mathcal{M}_{r}}(A)\cap T_{\mathcal{M}_{*}}(A))\leq 3(m+n)r-3r^2.
\end{equation}
 \end{theorem}
\begin{proof}
Recall Lemma \ref{q11} and Theorem \ref{q3}, we only need to show
\begin{equation*}
\dim N_{\mathcal{V}_{sr}}(A)=4mn-(4(m+n)r-4r^2)=4(mn-mr-nr+r^2)
\end{equation*}
if and only if $\rank(A)=4r.$ It follows Lemma \ref{q11} that $\dim N_{\mathcal{V}_{sr}}(A)\leq 4(mn-mr-nr+r^2)$, then it is sufficient to show
this inequality is strict if $\rank(A)<4r$ and the reverse inequality holds when $\rank(A)=4r.$
In this proof, the particular identification of $\Gamma$ given in Theorem  \ref{q3} with $\mathbb{R}^{4mn}$ is important.
Given a polynomial $p\in \mathbb{I}_{\mathcal{V}_{rs}}$ (where $\mathbb{I}_{\mathcal{V}_{rs}}$ is defined as the set of real polynomials that vanish on $\mathcal{V}_{rs}$) and two unitary matrices $U$ and $V$ with proper orders such that $q_{U,V}(X)=p(UXV^{T})$ is clearly also in $\mathbb{I}_{\mathcal{V}_{rs}}$.
Due to the particular choice of $\omega$, we have $\nabla_{q_{(U,V)}}(B)=U\nabla_{p}(UXV^{T})V^{T}$.
Let $A$ be fixed of $\rank (A)=j\leq 4r.$ Then there exist two unitary matrices $\hat{U}$ and $\hat{V}$ such that $\hat{U}A\hat{V}^{T}=S_{j}=\diag\{\Sigma,\Sigma,\Sigma,\Sigma\},$ where $\Sigma$  is a diagonal matrix whose diagonal are $\sigma_{i},i=1,..,j$ and 0 elsewhere.
It follows that $\nabla_{q_{(\hat{U},\hat{V})}}(A)=\hat{U}\nabla p(S_{j})\hat{V}^{T},$ which implies that $\dim N_{\mathcal{V}_{sr}}(A)=\dim N_{\mathcal{V}_{sr}}(S).$
Then all $\mathbb{R}_{4r+4,4r+4}$(in order to keep the structure of $\mathcal{V}_{sr}$) subdeterminants of $\Gamma$ form polynomials in $\mathbb{I}_{\mathcal{V}_{sr}}$. We can get
$\dim N_{\mathcal{V}_{sr}}(S)\geq 4(mn-mr-nr+r^2), $ then it prove that any matrix $X$ in the set $\Gamma$ with $\rank(X)=4r$ element of $\mathcal{V}_{sr}$ is non-singular. In other direction, if $j<r$, similarly as the above we can construct a variety with dimension of $4(mn-mj-nj+j^2)$ which is bigger than
$4(mn-mr-nr+r^2)$. Consider two fixed matrices $\widetilde{U}\in \mathbb{R}^{4m\times 4}$ and $\widetilde{V}\in \mathbb{R}^{4m\times 4}$ and define the map
$\theta_{\widetilde{U},\widetilde{V}}:\mathbb{R}^{4m\times 4}\rightarrow \mathcal{\widetilde{V}}_{sr}$ via $\theta_{\widetilde{U},\widetilde{V}}(x)=S+x\widetilde{U}\widetilde{V}^{T}.$
Then
\begin{equation*}
\spn\left\{\frac{d}{dx}\theta_{\widetilde{U},\widetilde{V}}(0):\widetilde{U},\widetilde{V}\in \mathbb{R}^{4n\times 4}\right\}=\Gamma,
\end{equation*} which is saying that the dimension of the differential geometry tangent space is $4mn$. Then $\dim N_{\mathcal{V}_{sr}}=0,$ and $S_{j}$ is singular. It follows that $\mathcal{V}_{sr}^{ns}=\mathcal{M}_{r}$.

Next, we will prove \eqref{neweq1} is satisfied.
Choose a point $A=US V^{T}\in \mathcal{M}_{r} $, where $U\in \mathbb{R}^{4m\times 4r},$ $S=\diag{\{\Sigma,\Sigma,\Sigma,\Sigma\}}$ and $V\in \mathbb{R}^{4n\times 4r}.$
Denote
\begin{align*}
\mathcal{T}_{i}=\left(
    \begin{array}{cc}
       \mathbb{R}^{r\times r} &\mathbb{R}^{r\times (n-r)} \\
      \mathbb{R}^{(m-r)\times r} & 0  \\
    \end{array}
  \right), i=0,1,2,3.
\end{align*}
By Lemma \ref{lem1}, the tangent space of $\mathcal{M}_{r}$ at $A$ can be expressed as
\begin{align*}
\mathcal{T}_{\mathcal{M}_{r}}(A)=\left\{[U,U_{\bot}]\left(
                                          \begin{array}{cccc}
                                            \mathcal{T}_{0} & -\mathcal{T}_{1} & \mathcal{T}_{2} & -\mathcal{T}_{3} \\
                                            \mathcal{T}_{1} & \mathcal{T}_{0} & -\mathcal{T}_{3} & \mathcal{T}_{2} \\
                                            -\mathcal{T}_{2} & \mathcal{T}_{3} & \mathcal{T}_{0} & -\mathcal{T}_{1} \\
                                            \mathcal{T}_{3} & -\mathcal{T}_{2} & \mathcal{T}_{1} & \mathcal{T}_{0} \\
                                          \end{array}
                                        \right)[V,V_{\bot}]^{T}
                                        \right\}.
\end{align*}
Then, it is easy to prove $T_{\mathcal{M}_{*}}(A)=\spn (W)$, with
\begin{align*}
W=\left(
                      \begin{array}{cccc}
                        0 & -\widetilde{\mathbf{1}}^{m\times n} & \widetilde{\mathbf{1}}^{m\times n}& -\widetilde{\mathbf{1}}^{m\times n} \\
                        \widetilde{\mathbf{1}}^{m\times n} & 0 & -\widetilde{\mathbf{1}}^{m\times n} & \widetilde{\mathbf{1}}^{m\times n} \\
                        -\widetilde{\mathbf{1}}^{m\times n} & \widetilde{\mathbf{1}}^{m\times n} & 0 & -\widetilde{\mathbf{1}}^{m\times n} \\
                        \widetilde{\mathbf{1}}^{m\times n}& -\widetilde{\mathbf{1}}^{m\times n} & \widetilde{\mathbf{1}}^{m\times n} & 0 \\
                      \end{array}
                    \right)~\text{with}~\widetilde{\mathbf{1}}=\left(
                            \begin{array}{cccc}
                              1 & 1 & \cdots & 1 \\
                              1 & 1 & \cdots & 1 \\
                              \vdots & \vdots & \ddots & \vdots \\
                              1 & 1 & \cdots & 1 \\
                            \end{array}
                          \right).
\end{align*}
After that we can obtain \eqref{neweq1}.

\end{proof}

 It follows from Theorem \ref{q3} that $\mathcal{V}$ is an irreducible variety with dimension $m=3(m+n)r-3r^2,$  then by Lemma \ref{new1}, $\mathcal{V}^{ns}$ can be chosen as a manifold with dimension $m=3(m+n)r-3r^2.$  Morover, the set of nonsingular points of $\mathcal{V}_{sr}$ forms the manifold $\mathcal{M}_{r}.$  Hence,  by Lemma \ref{lm1}, every intersction point of $\mathcal{M}_{r}\cap \mathcal{M}_{*}\cap \mathcal{V}^{ns}$ is a non-trivial intersection point, i.e., the angle between $\mathcal{M}_{r}$ and $\mathcal{M}_{*}$ is well defined. Denote $\mathcal{V}^{ns,nt}\subset \mathcal{V}$ as the set of all points in $\mathcal{M}_{r}\cap \mathcal{M}_{*}\cap \mathcal{V}^{ns}$ that are nontangential with respect to the manifolds $\mathcal{M}_{r}$ and $\mathcal{M}_{*}.$ Then by \eqref{neweq1} and Lemma \ref{lm2},  we have $\mathcal{V}^{ns,nt}\neq \emptyset, $ which is saying that nontangentiality at one single intersection point implies nontangentiality at all the points of the manifold.

Based on the above results, we can get the main results of this section.
\begin{theorem}\label{th2}
The set $\mathcal{V}^{ns,nt}=\mathcal{M}_{r}\cap \mathcal{M}_{*}=\mathcal{M}_{r*}$ is an $3(m+n)r-3r^2$ dimensional  manifold.
Its complement $\mathcal{V}\setminus \mathcal{V}^{ns,nt}$ is a finite set of connected manifolds of lower dimension.
\end{theorem}

\section{Alternating Projections on Manifolds}
\label{sec:main}

In this section,  the alternating projection method is chosen to solve the problem \eqref{pmain}.
The basic idea of alternating projections is to find a point in the intersection of two sets by iteratively projecting a point into one set and then the other.
Here, one manifold is the fixed rank $r$ quaternion matrix set $\mathcal{M}_{r}$  given as \eqref{v1},
 and the other one is the pure quaternion matrix set $\mathcal{M}_{*}$ given as \eqref{v2}.
It follows that \eqref{pmain} can be rewritten as finding the nearest matrix in the set $\mathcal{M}_{r*}$ given as \eqref{v3}, i.e., the intersection of the above two manifolds.

We first introduce two projections that project the given matrix onto the two matrix sets, respectively.
Similar to the real and complex matrix cases, the Eckart-Young-Mirsky low-rank approximation theorem \cite{golub2012matrix} still hold for quaternion matrices.
With the singular value decomposition of quaternion matrix given in Theorem \ref{t:qsvd}, the projection onto fixed rank matrix set $\mathcal{M}_{r}$ can be expressed as
\begin{align}\label{p1}
\pi_{1}({\bf X})=\sum_{i=1}^{r}\sigma_{i}({\bf X}){\bf u}_{i}({\bf X}){\bf v}_{i}^{T}({\bf X}),
\end{align}
where $\sigma_{i}({\bf X})$ are the $r$ first singular values of ${\bf X}$, and ${\bf u}_{i}({\bf X}),{\bf v}_{i}({\bf X})$ are the first $r$ columns of the unitary matrices of ${\bf U}$ and ${\bf V}$ given in Theorem \ref{t:qsvd}, respectively.
In addition, for an arbitrary quaternion matrix ${\bf X}=X_{0}+X_{1}{\bf i}+X_{2}{\bf j}+X_{3}{\bf k}\in \mathbb{H}^{m\times n}$, the projection onto the affine manifold $\mathcal{M}_{*}$ is exactly removing the real part, i.e.,
\begin{align}\label{p2}
\pi_{2}({\bf X})=\pi_{2}(X_{0}+X_{1}{\bf i}+X_{2}{\bf j}+X_{3}{\bf k})=X_{1}{\bf i}+X_{2}{\bf j}+X_{3}{\bf k}.
\end{align}
Recall that $\mathcal{M}_{r}$ and $\mathcal{M}_{*}$ are fix rank manifold and affine manifold introduced in Section \ref{sec:pre}.
Then the projection mappings may not be single valued.
We write $\pi_{1}({\bf X})$ and $\pi_{2}({\bf X})$ to denote an arbitrarily closest point to ${\bf X}$ on the manifolds  $\mathcal{M}_{r}$ and $\mathcal{M}_{*}$, respectively.

Nevertheless,  the projection onto the intersection  $\mathcal{M}_{r*}$ cannot be computed efficiently. We use $\pi ({\bf X})$ to denote an arbitrarily closest point to ${\bf X}$ on  the intersection  $\mathcal{M}_{r*}$. Furthermore, the convergence of the alternating projections cannot be guaranteed in general even when the two non-convex sets have a nonempty intersection, which is different from the convex case.
For instance, suppose $\textrm{K}=\mathbb{R}^{2}$ and denote $\mathcal{M}_{1}=\{t, (t+1)(3-t)/4:t\in \mathbb{R}\}$ and $\mathcal{M}_{2}=\mathbb{R}\times \{0\}.$ It easy to see that $\pi_{1}((1,0))=(1,1)$ and $\pi_{2}((1,1))=(1,0),$ then the sequence of alternating projections does not convergence.
Therefore, it is more difficult to consider alternating projections on non-linear manifolds than convex sets.

The following algorithm describes the alternating projections method.
of problem \eqref{pmain}.
\begin{algorithm}[!h]\label{ab1}
\caption{Alternating projections on manifolds} \label{alg1}
\textbf{Input: } Given a quaternion matrix ${\bf A}\in \mathbb{H}^{m\times n}$ this algorithm computes optimal rank-$r$ pure quaternion matrix approximation. \\
~~1: Initialize ${\bf X}_0 = {\bf A}$; \\
~~2: \textbf{for} $k=1,2,...$\\
~~3: \quad ${\bf Y}_{k+1}=\pi_{1}({\bf X}_{k});$\\
~~4: \quad ${\bf X}_{k+1}=\pi_{2}({\bf Y}_{k+1});$\\
~~5: \textbf{end}\\
\textbf{Output:} ${\bf X}_k$ when the stopping criterion is satisfied.
\end{algorithm}

Combining the above results with Theorem \ref{th2} and Theorem 5.1 in \cite{andersson2013alternating}, we can obtain the main result of this paper.
\begin{theorem}\label{thm_convergence}
 Let $\mathcal{M}_{r}$, $ \mathcal{M}_{*}$ and $\mathcal{M}_{r*}=\mathcal{M}_{r}\cap \mathcal{M}_{*}$ be given as \eqref{v1}- \eqref{v3}, respectively.
The projections onto the manifolds $\mathcal{M}_{r}$ and $\mathcal{M}_{*}$ are given in \eqref{p1} and \eqref{p2},  denote $\pi$ as the projection onto the manifold $\mathcal{M}_{r*}.$
Suppose that ${\bf B}\in \mathcal{M}_{r*}$
is a non-tangential intersection point of $\mathcal{M}_{r}$ and $\mathcal{M}_{*}$, then for any  given $\epsilon>0$ and $1>c>\sigma({\bf B})$, there exist an
$r>0$ such that for any ${\bf A}\in {\cal B}({\bf B},r)$
  the sequence $\{{\bf X}_k\}_{k=0}^\infty$ generated by the alternating projections algorithm initializing from ${\bf A}$ satisfies the following results:
  \begin{enumerate}[(1)]
    \item the sequence converges to a point ${\bf X}_\infty \in {\cal M}_r \cap {\cal M}_*$,
    \item $\| {\bf X}_\infty - \pi({\bf A}) \| \leq \epsilon \| {\bf A} - \pi({\bf A}) \|$,
    \item $\| {\bf X}_\infty - {\bf X}_k \| \leq {\rm const} \cdot c^k \| {\bf A} - \pi({\bf A}) \|$.
  \end{enumerate}
\end{theorem}

\section{Initialization for Alternating Projections Algorithm} \label{sec:initia}
Theoretically, Theorem \ref{thm_convergence} implies that the alternating projections algorithm linearly converges to a good approximation to $\pi({\bf A})$ assuming that ${\bf A}$ is in some neighborhood ${\cal N}_{\epsilon,c}$ of the intersection manifold.
Nevertheless, it is hard to check whether the original matrix is inside such a neighborhood or not since there is no explicit formula for the radius function $r_{\epsilon,c}$ in terms of the given scalars $\epsilon$ and $c$.
Therefore, it is necessary to design an initialization strategy to make the alternating projections method more practical and reliable.

Recall the target projection is as follows
\begin{equation}\label{eq_proj}
  \begin{array}{rl}
    \pi({\bf A}) \in {\rm arg}\min & \|{\bf X}-{\bf A}\|_\textrm{F}^2, \\
    {\rm s.t.} & {\rm Re}({\bf X}) = 0, \\
                & {\rm rank}({\bf X}) = r.
  \end{array}
\end{equation}
We aim at an initial point $X^0$ which is close to the intersection of the two manifolds and $\pi(X^0) \approx \pi(A)$.
If we apply a convergent iterative methods to the optimization problem \eqref{eq_proj}, then it is reasonable to regard an iterate after several steps as such a good initial point.
The reasons why we do not use this convergent iterative methods are (i) the convergence of this guaranteed method could be pretty slow and (ii) the computational cost of the alternating projections is generally much cheaper.

We reformulate the projection \eqref{eq_proj} to an unconstrained problem
\begin{equation}\label{eq_ind1}
  \min {\textstyle\frac{1}{2}}\|{\bf X}-{\bf A}\|_\textrm{F}^2 + \delta_{{\cal S}_0}({\bf X})
  + {\textstyle\frac{1}{2}}\|{\bf X}-{\bf A}\|_\textrm{F}^2 + \delta_{{\cal S}_r}({\bf X}),
\end{equation}
where $\delta_C$ denotes the indicator function of the set $C$.
To adapt the convergence conditions which will be discussed shortly, we further relax the problem \eqref{eq_ind1} to
\begin{equation}\label{eq_ind2}
  \min \underbrace{{\textstyle\frac{1}{2}}\|{\bf X}-{\bf A}\|_\textrm{F}^2 + {\textstyle\frac{\tau}{2}}\| {\rm Re}({\bf X}) \|_\textrm{F}^2}_{f({\bf X})}
  + \underbrace{{\textstyle\frac{1}{2}}\|{\bf X}-{\bf A}\|_\textrm{F}^2 + \delta_{{\cal S}_r}({\bf X})}_{g({\bf X})}.
\end{equation}
Note that the problems \eqref{eq_ind1} and \eqref{eq_ind2} are equivalent when $\tau$ approaches the infinity.

Li and Pong \cite{LP16} proposed and investigate the Douglas-Rachford splitting method (DRSM) for solving the nonconvex optimization problem
\begin{equation*}
  \min f(x) + g(x).
\end{equation*}
The DRSM iterates
\begin{equation*}
  \left\{
  \begin{array}{l}
    y^{k+1} \in {\rm prox}_{\alpha f}(x^k) := {\rm arg}\min\limits_y \big\{ f(y) + \frac{1}{2\alpha}\|y-x^k\|_2^2 \big\}, \\
    z^{k+1} \in {\rm prox}_{\alpha g}(2y^{k+1}-x^k) := {\rm arg}\min\limits_z \big\{ g(z) + \frac{1}{2\alpha}\|2y^{k+1}-x^k-z\|_2^2 \big\}, \\
    x^{k+1} = x^k + (z^{k+1} - y^{k+1}),
  \end{array}
  \right.
\end{equation*}
Assuming the existence of a cluster point, they proved the following conditions can guarantee the global convergence to a stationary point:
\begin{enumerate}[(1)]
  \item $f$ has a Lipschitz continuous gradient whose Lipschitz continuity modulus is bounded by $L$,
  \item $g$ is a proper closed function,
  \item $f$ and $g$ are semi-algebraic functions,
  \item $0 < \alpha < \frac{1}{L}$.
\end{enumerate}

In the problem \eqref{eq_ind2}, the function $f({\bf X})$ is quadratic and thus it is semi-algebraic and has Lipschitz continuous gradient whose Lipschitz constant is $1+\tau$.
The manifold ${\cal S}_r$ can be characterized by $$\{ {\bf Y}:\, \det({\bf Y}_{r+1})=0 \text{ for any $(r+1)$-by-$(r+1)$ submatrix } {\bf Y}_{r+1} \}. $$
Hence, this set is a semi-algebraic set, which implies that its indicator function is a semi-algebraic function.
That is, $g({\bf X})$ is also a semi-algebraic function.
Therefore, the global convergence to a stationary point can be guaranteed as long as we choose the stepsize $\alpha$ less than $\frac{1}{1+\tau}$.

Furthermore, the proximal operator for $\alpha f$ is implemented by
\begin{equation*}
\begin{split}
  {\rm prox}_{\alpha f}({\bf Y}) &=
  (\textstyle\frac{\alpha}{1+\alpha+\alpha\tau}A_1 + \frac{1}{1+\alpha+\alpha\tau}Y_1) +
  (\textstyle\frac{\alpha}{1+\alpha}A_2 + \frac{1}{1+\alpha}Y_2) {\bf i} \\&\quad +
  (\textstyle\frac{\alpha}{1+\alpha}A_3 + \frac{1}{1+\alpha}Y_3) {\bf j} +
  (\textstyle\frac{\alpha}{1+\alpha}A_4 + \frac{1}{1+\alpha}Y_4) {\bf k},
\end{split}
\end{equation*}
and the proximal operator for $\alpha g$ is
\begin{equation*}
  {\rm prox}_{\alpha g}({\bf Y}) = \pi_1 \big( \textstyle\frac{\alpha}{1+\alpha}{\bf A} + \frac{1}{1+\alpha}{\bf Y} \big),
\end{equation*}
i.e., the truncation to a rank-$r$ quaternion matrix.
To sum up, the DRSM for solving \eqref{eq_ind2} iterates
\begin{equation}\label{eq_drsm_initial}
  \left\{
  \begin{array}{l}
    {\bf Y}^{k+1} = (\textstyle\frac{\alpha}{1+\alpha+\alpha\tau}A_1 + \frac{1}{1+\alpha+\alpha\tau}X^k_1) +
  (\textstyle\frac{\alpha}{1+\alpha}A_2 + \frac{1}{1+\alpha}X^k_2) {\bf i} \\
  \hspace{40pt} +
  (\textstyle\frac{\alpha}{1+\alpha}A_3 + \frac{1}{1+\alpha}X^k_3) {\bf j} +
  (\textstyle\frac{\alpha}{1+\alpha}A_4 + \frac{1}{1+\alpha}X^k_4) {\bf k}, \\
    {\bf Z}^{k+1} \in \pi_1 \big( \textstyle\frac{\alpha}{1+\alpha}{\bf A} + \frac{2}{1+\alpha}{\bf Y}^{k+1}-\frac{1}{1+\alpha}{\bf X}^k \big), \\
    {\bf X}^{k+1} = {\bf X}^k + ({\bf Z}^{k+1} - {\bf Y}^{k+1}).
  \end{array}
  \right.
\end{equation}
We perform the DRSM for a fixed number steps and then apply the generated iterate $Y^k$ as the initial point of the alternating projections method.
Note that this initialization strategy is still heuristic and the improvement using this method will be numerically illustrated in the following section.


\section{Numerical Experiments}\label{sec:exper}

In this paper, we focus on searching an optimal rank-$r$ pure quaternion matrix approximation of a given quaternion matrix.
Although this problem  is  difficult,  there exist some suboptimal methods to solve it.
For example,  in \cite{jia2013new}, the authors  do the rank $r$ truncation of a given quaternion matrix and then take three imaginary parts as an approximation of the given quaternion matrix. This method is called   ``QsvdTr''  in the sequel.
However, if the real part of the quaternion matrix is removed, the rank of the quaternion matrix generally changes.  This fact can be guaranteed by the following proposition (The proof can be found in the Appendix).

\begin{proposition}\label{rankincrease}
 For an arbitrary quaternion ${\bf A}=A_{0}+A_{1}{\bf i}+A_{2}{\bf j}+A_{3}{\bf k}\in 
\mathbb{H}^{m\times n}$, with $\rank({\bf A})=r \leq \min\{m,n\}/4$, we denote the 
pure quaternion part of ${\bf A}$ as ${\bf A}_{p}=A_{1}{\bf i}+A_{2}{\bf j}+A_{3}{\bf k},$ then
$$r\leq \rank({\bf A}_{p})\leq 4r.$$
\end{proposition}
In Subsection \ref{ssec1}-\ref{ssec3}, the performances of  the ``QsvdTr'' algorithm given in \cite{jia2013new} and  the ``AltProj'' algorithm  proposed in Algorithm \ref{pmain} are compared by testing  synthetic data, random data and color images, respectively.  We use the running time  and the objective function values, i.e., $\|{\bf X}^k-{\bf A}\|_F$ to compare the results derived by ``QsvdTr'' and ``AltProj''.
All the experiments are performed under Windows 10 and MATLAB R2018a running on a desktop (Intel Core i7, @ 3.40GHz, 8.00G RAM).

\subsection{Synthetic Data}\label{ssec1}

In our first example, we compare the ``AltProj''  algorithm and   ``QsvdTr''  algorithm by finding the optimal rank 4 pure quaternion 
approximation of the pure quaternion matrix
{\scriptsize
\begin{align*}
{\bf A}  &  =  \left(
    \begin{array}{ccccc}
      0.37 & -0.79 & 0.04 & -0.73 & -0.06 \\
      -1.42 & -0.10 & 1.01 & 1.59 & -1.59 \\
      -0.34 & 0.38 & 1.30 & -0.66 & 1.08 \\
      -1.98 & 0.83 & 0.22 & -0.77 & 0.70 \\
      -0.38 & -0.14 & 0.86 & 0.54 & 1.65 \\
    \end{array}
  \right) {\bf i} +\left(
             \begin{array}{ccccc}
               0.29 & -0.38 & -0.13 & -1.77 & 0.20 \\
               0.70 & -0.69 & 0.83 & -0.16 & -0.52 \\
               -1.15 & 1.00 & -1.97 & 0.63 & 1.57 \\
               1.86 & -1.14 & 0.12 & -1.27 & 0.77 \\
               2.37 & 0.15 & 0.26 & -0.30 & -0.59 \\
             \end{array}
           \right) {\bf j}  \nonumber \\
      &   +   \left(
                      \begin{array}{ccccc}
                        0.33 & 0.74 & -1.40 & -0.77 & 0.86 \\
                        1.13 & -1.32 & 0.36 & -0.02 & 0.50 \\
                        0.25 & -0.68 & 0.36 & -0.71 & 0.77 \\
                        0.56 & -0.35 & 0.92 & 0.87 & -0.58 \\
                        0.64 & -1.59 & 0.37 & -1.51 & 0.19 \\
                      \end{array}
                    \right) {\bf k}.
\end{align*}}
By applying Algorithm \ref{alg1} on ${\bf A}$, we can get a  rank 4  pure quaternion approximation as
{\scriptsize
\begin{align*}
{\bf A}_{4}&=\left(
    \begin{array}{ccccc}
      0.50 & -0.73 & -0.01& -0.68 & 0.08 \\
      -1.34 & -0.11 & 1.05 & 1.69 & -1.50 \\
      -0.29 & 0.36 & 1.24 & -0.64 & 1.11 \\
      -2.01 & 0.81 & 0.11 & -0.83 & 0.56 \\
      -0.37 & -1.12 & 0.94 & 0.57 & 1.70 \\
    \end{array}
  \right){\bf i}+\left(
             \begin{array}{ccccc}
               0.23 & -0.50 & 0.08 & -1.71 & 0.37 \\
               0.65 & -0.95 & 0.76 & -0.17 & -0.47 \\
               -1.17 & 0.91 & -1.96 & 0.62 & 1.63 \\
               1.92 & -1.01 & 0.05 & -1.30 & 0.74 \\
               2.35 & 0.09 & 0.28 & -0.28 & -0.58 \\
             \end{array}
           \right){\bf j}\\
   & +\left(
              \begin{array}{ccccc}
                0.39 & 0.65 & -1.41 & -0.66 & 0.92 \\
                1.07 & -1.34 & 0.34 & -0.04 & 0.36 \\
                0.22 & -0.76 & 0.37 & -0.70 & 0.77 \\
                0.51 & -0.44 & 0.92 & 0.82 & -0.54 \\
                0.67 & -1.53 & 0.35 & -1.50 & 0.16 \\
              \end{array}
            \right){\bf k}.
\end{align*}
}
The singular value decomposition of ${\bf A}_{4}$ can be expressed as
${\bf A}_{4}={\bf U}S{\bf V}$ with
{\scriptsize
\begin{align*}
{\bf U}&=\left(
        \begin{array}{ccccc}
          -0.02 & -0.11 & 0.02 & 0.17 & -0.43 \\
          -0.10 & -0.24 & -0.04 & -0.19 &0.07 \\
          0.21 & 0.14& -0.34 & 0.17 & 0.02 \\
          0.06 & 0.24 & 0.16 & 0.30 & 0.36 \\
          0.04 & -0.01 & 0.10 & 0.04 & -0.20 \\
        \end{array}
      \right)+\left(
                \begin{array}{ccccc}
                  -0.27 & -0.14 & 0.30 & 0.25 & 0.25 \\
                  0.35 & 0.29 & 0.14 & -0.37 & 0.28 \\
                  -0.05 & -0.50 & -0.16 & -0.05 & 0.18 \\
                  0.33 & -0.30 & 0.18 & -0.05 & -0.11 \\
                  0.08 & -0.05 & 0.12 & 0.34 & -0.03 \\
                \end{array}
              \right){\bf i}\\
             & +\left(
                 \begin{array}{ccccc}
                   -0.15 & -0.05 & 0.13 & -0.10 & 0.16 \\
                   -0.19 & -0.12 & -0.10 & -0.14 & 0.46 \\
                   0.22 & -0.04 & 0.11 & -0.22 & 0.17 \\
                   -0.40 & 0.02 & -0.32 & -0.25& -0.08 \\
                   -0.32 & 0.53 & -0.22 & 0.06 & 0.05 \\
                 \end{array}
               \right){\bf j}+\left(
                          \begin{array}{ccccc}
                            -0.05 & -0.32 & -0.34 & 0.18 & 0.37 \\
                            -0.31 & -0.07 & 0.08 & 0.17& -0.17 \\
                            -0.23 & 0.01 & 0.49 & 0.21 & 0.07 \\
                            0.00 & 0.03 & -0.23 & 0.23 & 0.07 \\
                            -0.33 & 0.11 & 0.24 & -0.44 & -0.09 \\
                          \end{array}
                        \right){\bf k},
                        \end{align*}

                      \begin{align*}
            {\bf V} &=\left(
                      \begin{array}{ccccc}
                        0.65 & 0.31 & -0.54 & 0.30 & 0.32 \\
                        0.39 & -0.08 & -0.24 & 0.18 & 0.31 \\
                        -0.09 & 0.05 & 0.10 & 0.18 & 0.31 \\
                        0.39 & 0.28 & -0.10 & 0.24 & 0.13 \\
                        -0.06 & -0.43 & 0.12 & 0.41 & 0.10 \\
                      \end{array}
                    \right)+
                    \left(
                      \begin{array}{ccccc}
                        0 & 0 & 0 & 0 & 0 \\
                        -0.12 & 0.12 & 0.06 & -0.48 & 0.18 \\
                        0.03 & -0.06 & -0.07 & -0.16 & 0.38 \\
                        0.05 & 0.41 & 0.42 & 0.27 & 0.16 \\
                        -0.11 & -0.14 & 0.15 & 0.11 & 0.04 \\
                      \end{array}
                    \right){\bf i}\\
              & +     \left(
                      \begin{array}{ccccc}
                        0 & 0 & 0 & 0 & 0 \\
                        -0.19 & -0.07 & 0.04 & 0.05 & -0.29 \\
                        0.38 & 0.07 & -0.59 &0.10&-0.28 \\
                        0.14 & 0.04 & -0.02 & 0.18 & 0.03 \\
                        -0.06 & -0.19 & -0.08 & 0.06 & -0.13 \\
                      \end{array}
                    \right){\bf j}+\left(
                               \begin{array}{ccccc}
                                 0 & 0 & 0 & 0 & 0 \\
                                 0.11 & 0.07 & -0.03 & -0.28 & 0.37 \\
                                 -0.13 & -0.26 & -0.22 & -0.22 & -0.18 \\
                                 0.00 & -0.07 & 0.07 & 0.39 & -0.16 \\
                                 -0.04 & 0.55 & 0.08 & -0.04 & -0.43 \\
                               \end{array}
                             \right){\bf k},
\end{align*}}
and
{\scriptsize
 \begin{align*}
               S&=\left(
                       \begin{array}{ccccc}
                         5.90 & 0 & 0 & 0 & 0 \\
                         0 & 4.47 & 0 & 0 & 0 \\
                         0 & 0 & 2.96 & 0 & 0 \\
                         0 & 0 & 0 & 1.75 & 0 \\
                         0 & 0 & 0 & 0 & 0 \\
                       \end{array}
                     \right).
                      \end{align*}}
Moreover, by  Theorem \ref{t:qsvd} the optimal rank 4 approximation of ${\bf A}$ can be expressed as $\tilde{{\bf A}}=\tilde{{\bf A}}_{re}+\tilde{{\bf A}}_{\text{p}},$ where
{\scriptsize
\begin{align*}
\tilde{{\bf A}}_{\text{re}}&=\left(
                 \begin{array}{ccccc}
                   0.03 & 0.17 & 0.10 & 0.03 & -0.02 \\
                   -0.05 & -0.03 & 0.14 & 0.03 & 0.06 \\
                   0.01 & 0.02 & 0.06 & 0.03 & 0.04 \\
                   0.03 & -0.07 & -0.07& 0.02& 0.03 \\
                   -0.03 & 0.02 & 0.03 & -0.02 & -0.02 \\
                 \end{array}
               \right),\\
\tilde{{\bf A}}_{\text{p}}&=\left(
                 \begin{array}{ccccc}
                   0.46 & -0.75 & 0.03 & -0.70 & 0.04 \\
                   0.06 & -1.36 & -0.10 & 1.66 & -1.53 \\
                   -0.04 & -0.30 & 0.37 & 1.26 & -0.65 \\
                   0.03 & -2.00 & 0.82 & 0.14 & -0.82 \\
                   -0.02 & -0.37 & -0.13 & 0.91 & 0.56 \\
                 \end{array}
               \right){\bf i}+
               \left(
                 \begin{array}{ccccc}
                   0.25 & -0.47 & 0.01 & -1.73 & 0.31 \\
                   0.67 & -0.86 & 0.78 & -0.16 & -0.48 \\
                   -1.17 & 0.95 & -1.96 & 0.63 & 1.62 \\
                   1.91 & -1.05 & -0.08 & -1.29 & 0.75 \\
                   2.35 & 0.11 & 0.27 & -0.29 & -0.58 \\
                 \end{array}
               \right){\bf j}\\
               &+\left(
                   \begin{array}{ccccc}
                     0.38 & 0.69 & -1.42 & -0.70 & 0.89 \\
                     1.09 & -1.34 & 0.34 & -0.03 & 0.40 \\
                     0.23 & -0.73 & 0.36 & -0.71 & 0.77 \\
                     0.53 & -0.41 & 0.92 & 0.84 & -0.55 \\
                     0.66 & -1.55 & 0.36 & -1.51 & 0.17 \\
                   \end{array}
                 \right){\bf k}.
\end{align*}}
In this case, $\tilde{{\bf A}}_{\text{p}}$ is often chosen as the optimal 
pure quaternion rank $4$ approximation. However, $\rank(\tilde{{\bf A}}_{\text{p}})=5$, i.e., $\tilde{{\bf A}}_{\text{p}}$ is essentially not a rank $4$ approximation. Recall  Proposition \ref{rankincrease}  that the rank of a quaternion matrix will increase if  its real part is removed.
Then if one want to find a  rank $r$ 
pure quaternion
approximation, less than $r$ truncation of ${\bf A}$ is suitable.
So, we can get a  rank 4 pure quaternion
approximation {\scriptsize
\begin{align*}
{\bf A}_{\text{p}}&=
  \left(
    \begin{array}{ccccc}
      1.05 & -0.78 & 0.37 & -0.59 & 0.13 \\
      -1.32 & 0.25 & 0.69 & 1.07 & -0.26 \\
      0.20 & -0.06 & 0.31 & 0.03 & -0.01 \\
      -1.30 & 0.58 & 0.16 & 0.73 & 0.02 \\
      -0.27 & -0.38 & 0.98 & 0.42 & -0.03 \\
    \end{array}
  \right){\bf i}+\left(
     \begin{array}{ccccc}
       0.56 & -0.14 & -0.15 & -0.32 & 0.00 \\
       0.68 & -0.44 & 0.54 & -0.42 & 0.33 \\
       -0.78 & -0.93 & -0.65 & 0.24 & 0.12 \\
       1.52 & -1.07 & 0.39 & -0.97 & 0.28 \\
       1.28 & -0.47 & 0.08 & -0.88 & 0.37 \\
     \end{array}
   \right){\bf j}\\
   &+\left(
              \begin{array}{ccccc}
                0.17 & 0.09 & -0.57 & -0.30 & 0.04 \\
                1.21 & -1.20 & 0.86 & -0.40 & -0.14 \\
                0.87 & -0.42 & 0.17 & -0.61 & 0.30 \\
                -0.05 & -0.72 & 0.84 & 0.40 & -0.36 \\
                1.26 & -1.14 & 0.41 & -0.63 & -0.08 \\
              \end{array}
            \right){\bf k},
\end{align*}}
by deleting the real part of its optimal rank 1 approximation of ${\bf A}$. In this case, the objective function values are $\|{\bf A}-{\bf A}_{4}\|_{F}=0.6479$ and $\|{\bf A}-{\bf A}_{\text{p}}\|_{F}=5.3579,$ respectively.  We see that approximation derived by the proposed ``AltProj'' algorithm is better than that  derived by the ``QsvdTr" algorithm.

\subsection{Random Matrices}\label{ssec2}
 In our second experiment, we use random low rank 
pure quaternion matrix to illustrate the validity of Algorithm \ref{pmain}.  In the proposed `AltProj'' algorithm, the maximum number of iterations is chosen as $5000$ steps. And the iterations stops  when the residual, i.e., the Frobenius norms of the real part is less than $10^{-6}$. Since it is hard to generate a random low rank pure quaternion matrix directly, the following alternative method is applied.  We first generate m-by-n quaternion  matrices ${\bf A}=A_{0}+A_{1}{\bf i} +A_{2}{\bf j} +A_{3}{\bf k},$ where the matrix entries of $A_{i}, i=0,1,2,3$ follow  the standard normal distribution. Random quaternion matrices with ranks $1,2,3,4,5$ and $10$ can be derived by applying 
QSVD truncations on ${\bf A}$, respectively. It follows Proposition \ref{rankincrease} that the rank of quaternion matrix will increase when the real part is deleted. In addition, when the quaternion matrix is randomly generated with small rank, the columns of the different full rank decompositions with three kinds of  conjugate definitions are always independents then  $\rank({\bf A}_{p})=4\rank({\bf A})$ holds at most times.
 Although this result cannot be guaranteed to be correct in general, we often use this method to obtain approximate solutions of some problems in practical application. Then we can get ranks $4,8,12,16,20$ 
and $40$ pure quaternion matrices  by  setting the real parts of the these low rank quaternion matrices to be zeros, respectively.

Tables \ref{table1} shows the running times  and the objective function values  of the computed solutions from the proposed ``AltProj'' algorithm and the ``QsvdTr'' algorithm for ranks $4,8,12,16,20$ and $40$ random 
pure quaternion matrices sets of sizes 100-by-100, 200-by-200 and 500-by-500, respectively.  When the input quaternion matrix ${\bf A}$ is exactly a low rank 
pure quaternion matrix, the proposed ``AltProj'' algorithm can provide exact recovery results in the first iteration. However, there is no guarantee that the ``QsvdTr'' algorithm  can determine the 
low rank pure quaternion matrix. In the
tables, it is clear that the ``QsvdTr'' algorithm  cannot obtain the underlying low
rank factorization. The running times of the ``AltProj'' algorithm is nearly two times of ``QsvdTr'' algorithm.

\begin{table}[!h]\label{table1}
\centering
\caption{ The running times (Time) and the objective function values (OBF-value) by different algorithms for low rank pure quaternion approximations of random low rank 
pure quaternion matrices.} \label{table1}
\begin{tabular}{|c|c|cccccc|}
\hline
        \multirow{2}{*}{method} & \multirow{2}{*}{item} &\multicolumn{6}{c|}{100-by-100 quaternion matrix} \\
  \cline{3-8}       &                      &r=4         &r=8        &r=12       &r=16        &r=20        &r=40\\\hline

     \multirow{2}{*}{QsvdTr}           & OBF-value        &8.89        &12.78      &15.61    &18.05       &20.22        &28.03    \\\cline{2-8}
                                       & Time (s)  &0.28        &0.28      &0.29      &0.29       &0.29        &0.30 \\\hline
     \multirow{2}{*}{AltProj}          & OBF-value        &4.09e-14    &1.65e-13   &6.80e-13 &1.90e-13    &1.30e-13     &1.87e-13\\\cline{2-8}
                                       & Ttime (s)  &0.56        &0.57       &0.57     &0.58      &0.57         &0.58\\\hline
    \hline
        \multirow{2}{*}{method} & \multirow{2}{*}{item} &\multicolumn{6}{c|}{200-by-200 quaternion matrix} \\
  \cline{3-8}       &                      &r=4         &r=8        &r=12       &r=16        &r=20        &r=40\\\hline

     \multirow{2}{*}{QsvdTr}        & OBF-value               &12.58         &17.79      &21.91    &25.31       &28.24        &40.22 \\\cline{2-8}
                                    & Time (s)             & 1.40        &1.42      &1.41      &1.42       &1.46         &1.44 \\\hline
     \multirow{2}{*}{AltProj}        & OBF-value              &6.07e-14     &1.28e-13   &2.86e-13  &3.40e-13    &2.25e-13     &3.21e-13\\\cline{2-8}
                                     &  Time (s)          &2.79         &2.80       &2.83      &2.84        &2.83       &2.85\\\hline
  \hline
        \multirow{2}{*}{method} & \multirow{2}{*}{item} &\multicolumn{6}{c|}{500-by-500 quaternion matrix} \\
  \cline{3-8}       &                      &r=4         &r=8        &r=12       &r=16        &r=20        &r=40\\\hline

     \multirow{2}{*}{QsvdTr}        & OBF-value            &19.57        &27.61      &33.98     &39.07       &43.92        &62.39 \\\cline{2-8}
                                   &  Time (s)         & 19.11       &18.94      &19.03      &19.13       &19.11        &19.27 \\\hline
     \multirow{2}{*}{AltProj}        & OBF-value            &9.50e-14     &2.70e-13   &3.19e-13 &6.13e-13    &4.66e-13     &6.57e-13\\\cline{2-8}
                                    & Time (s)         &38.06        &38.14       &38.45    &38.48      &38.21        &38.48\\\hline
\end{tabular}
\end{table}

In our third experiment, we use random pure quaternion matrix to compare the two algorithms, where the low rank minimizer is unknown in this setting.   The maximum number of iterations and the tolerance of the residual relate to the ``AltProj'' algorithm are chosen as $5000$ steps and $10^{-6}$, respectively. We randomly generate $m$-by-$n$ pure quaternion  matrices ${\bf A}=A_{1}{\bf i} +A_{2}{\bf j} +A_{3}{\bf k},$ where the matrix entries of $A_{i}, i=1,2,3$ follow  the standard normal distribution.   Then by applying the  ``QsvdTr" algorithm  and the proposed ``AltProj" algorithm  to  ${\bf A},$ we can find its optimal ranks $4,8,12,16,20$ and $40$ pure quaternion approximations, respectively. Table \ref{ntable1} shows that the running times and the objective function values
of the computed solution ${\bf X}^{k}$ from the proposed ``AltProj'' algorithm and the ``QsvdTr'' algorithm.
We see from Table \ref{ntable1} that the objective function values  computed by the proposed
``AltProj''  algorithm are smaller than that derived  by the testing ``QsvdTr'' algorithm, although the corresponding running times are longer.

\begin{table}[!h]\label{ntable1}
\centering
\caption{ The running times (Time) and the objective function values (OBF-value) by different algorithms for low rank pure quaternion approximations of random pure quaternion matrices.} \label{ntable1}
\begin{tabular}{|c|c|cccccc|}
\hline
        \multirow{2}{*}{method} & \multirow{2}{*}{item} &\multicolumn{6}{c|}{100-by-100 quaternion matrix} \\
  \cline{3-8}       &                                     &r=4         &r=8        &r=12       &r=16        &r=20        &r=40\\\hline
     \multirow{2}{*}{QsvdTr}        & OBF-value             &171.00        &166.97      &163.10    &159.43       &155.78        &139.52 \\\cline{2-8}
                                     & Time (s)         & 0.32         &0.33         &0.32    &0.31           &0.32        &0.31 \\\hline
     \multirow{2}{*}{AltProj}        & OBF-value            &168.48       &160.66       &152.29    &143.66        &1.35.24     &91.12\\\cline{2-8}
                                     &Time (s)       &1.73     &33.36   &2.83 &48.37      &24.01  &12.70\\\hline
    \hline
        \multirow{2}{*}{method} & \multirow{2}{*}{item} &\multicolumn{6}{c|}{200-by-200 quaternion matrix} \\
  \cline{3-8}       &                      &r=4         &r=8        &r=12       &r=16        &r=20        &r=40\\\hline

     \multirow{2}{*}{QsvdTr}           &OBF-value            &342.33        &338.21      &334.31    &330.57     &326.79      &308.86 \\\cline{2-8}
                                       &Time (s)       &1.56          &1.55        &1.51      &1.55       &1.56        &1.55 \\\hline
     \multirow{2}{*}{AltProj}          & OBF-value            &339.88     &332.94   &325.36 &317.80   &310.34     &268.59\\\cline{2-8}
                                       & Time (s)         &13.42     &11.71   &11.31 &14.60      &29.84  &30.48\\\hline
  \hline
        \multirow{2}{*}{method} & \multirow{2}{*}{item} &\multicolumn{6}{c|}{500-by-500 quaternion matrix} \\
  \cline{3-8}       &                      &r=4         &r=8        &r=12       &r=16        &r=20        &r=40\\\hline

     \multirow{2}{*}{QsvdTr}           &OBF-value            &861.25        &857.10      &853.01    &848.98       &844.98        &825.60 \\\cline{2-8}
                                       &Time (s)        & 20.58       &20.67      &20.32    &20.28       &20.48        &21.75 \\\hline
     \multirow{2}{*}{AltProj}          & OBF-value            &860.02     &853.28   &846.58 &839.92    &823.73     &794.54\\\cline{2-8}
                                       & Time (s)         &72.04     &184.56   &175.56 &1460.45      &405.81  &1099.45\\\hline
\end{tabular}
\end{table}

\subsection{Color Images}\label{ssec3}
In this subsection,  we employ the color images `peppafamily'', ``pepper'' and ``colortexture'' with sizes $200$-by-$200$
to compare ``QsvdTr'' algorithm and ``AltProj''algorithm  in terms of objective function values and time.
The maximum number of iterations and the tolerance of the residual relate to the ``AltProj'' algorithm are chosen as $5000$ steps and $10^{-6}$, respectively.
The original three color images employed in this subsection are shown in the first column of Figure 1, and the rank $16$ and $20$ pure quaternion approximations by the ``QsvdTr'' algorithm are listed in the second and fourth columns, the optimal rank $16$ and $20$ pure quaternion
approximations derived by our algorithm are listed in the third and fifth columns. For the two cases, we can see respectively that the images derived by the  ``AltProj'' algorithm  are better than those derived by ``QsvdTr'' algorithm in terms  of visual quality. The  ``AltProj'' algorithm can preserve more details than
``QsvdTr'' algorithm for the three testing images.

Moreover, we also compute the objective function values of  ranks $4, 8, 12, 16, 20$ and $40$ approximations, respectively, which illustrates the validity of our method. The results and shown in Table \ref{table5}. It can be seen
that the  objective function values  obtained by ``AltProj" algorithm are much lower
than those by ``QsvdTr'' algorithm. For the ``colortexture'' and ``pepper'' images, the  objective function values  of the ``AltProj" algorithm are nearly half of that derived by ``QsvdTr'' algorithm.
\begin{figure}[!h] \label{f1}
\centering
\begin{tabular}{ccccc}
\includegraphics[width=0.85in]{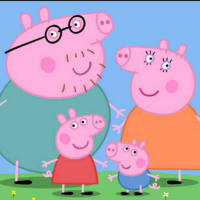} &
\includegraphics[width=0.85in]{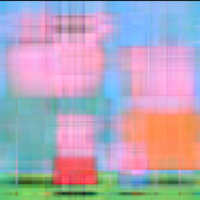}&
\includegraphics[width=0.85in]{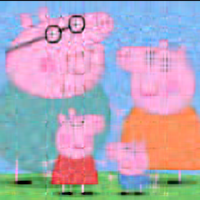}&
\includegraphics[width=0.85in]{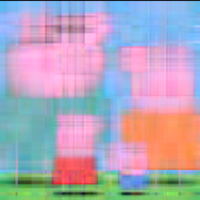}&
\includegraphics[width=0.85in]{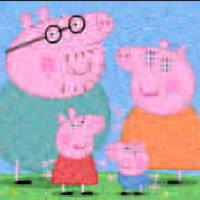}\\
\includegraphics[width=0.85in]{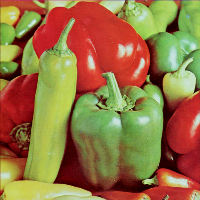} &
\includegraphics[width=0.85in]{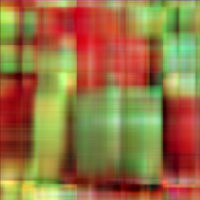}&
\includegraphics[width=0.85in]{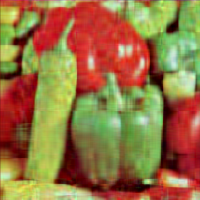}&
\includegraphics[width=0.85in]{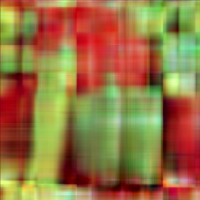}&
\includegraphics[width=0.85in]{opep4.png}\\
\includegraphics[width=0.85in]{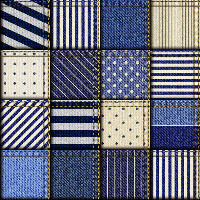}&
\includegraphics[width=0.85in]{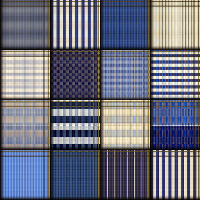} &
\includegraphics[width=0.85in]{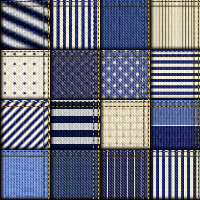}&
\includegraphics[width=0.85in]{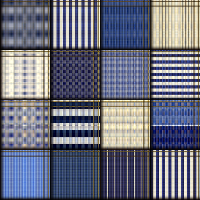} &
\includegraphics[width=0.85in]{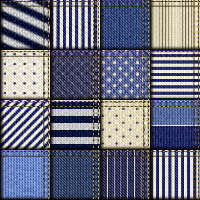}\\
original graph&$\textrm{r}=16,$ QsvdTr  & $\textrm{r}=16,$ AltProj  & $\textrm{r}=20,$ QsvdTr & $\textrm{r}=20,$ AltProj
\end{tabular}
\caption{ ranks 16 and 20 approximations of ``peppafamily'', ``pepper'' and ``colortexture'' by QsvdTr and AltProj algorithms, respectively}\label{fig6}
\end{figure}
\begin{table}[!h]
\centering
\caption{ The running times (Time) and the objective function values (OBF-value)  by different algorithms for low rank  approximations of  color images} \label{table5}
\begin{tabular}{|c|c|cccccc|}
    \hline
        \multirow{2}{*}{method} & \multirow{2}{*}{item} &\multicolumn{6}{c|}{peppafamily} \\
  \cline{3-8}       &                      &r=4         &r=8        &r=12       &r=16        &r=20        &r=40\\\hline

     \multirow{2}{*}{QsvdTr}           &OBF-value            &13731.6    &10944.1    &9450.5    &8742.4    &8143.7      &6470.5  \\\cline{2-8}
                                       &Time (s)        &1.55          &1.55        &1.51      &1.55       &1.56        &1.55 \\\hline
     \multirow{2}{*}{AltProj}          & OBF-value            &9077.2     &7351.9     &6218.8    &5439.9    &4770.3      &2726.0\\\cline{2-8}
                                       & Time (s)         &189.78     &189.06   &191.38 &190.19      &195.46  &211.15\\\hline
  \hline
        \multirow{2}{*}{method} & \multirow{2}{*}{item} &\multicolumn{6}{c|}{pepper} \\
  \cline{3-8}       &                      &r=4         &r=8        &r=12       &r=16        &r=20        &r=40\\\hline

     \multirow{2}{*}{QsvdTr}           &OBF-value            &17831.5     &14407.2    &12587.8   &11147.5   &10153.1     &7185.0  \\\cline{2-8}
                                       &Time (s)       & 1.67       &1.57      &1.68    &1.56       &1.59        &1.56 \\\hline
     \multirow{2}{*}{AltProj}          & OBF-value            &11392.7     &8494.4     &6804.4    &5667.4    &4949.5      &2923.9\\\cline{2-8}
                                       & Time (s)         &201.76     &206.77   &221.51 &206.21      &193.88  &199.59\\\hline
  \hline
        \multirow{2}{*}{method} & \multirow{2}{*}{item} &\multicolumn{6}{c|}{colortexture} \\
  \cline{3-8}       &                                     &r=4         &r=8        &r=12       &r=16        &r=20        &r=40\\\hline
     \multirow{2}{*}{QsvdTr}        & OBF-value             &21964.6   &18261.4   &15485.2   &14034.7  &13155.7     &9530.4  \\\cline{2-8}
                                     & Time (s)        & 1.63         &1.57         &1.57    &1.63           &1.85        &1.55 \\\hline
     \multirow{2}{*}{AltProj}        & OBF-value            &14039.5   &10864.6   &8583.1    &7199.8   &6527.4      &4496.8 \\\cline{2-8}
                                     & Time (s)       &219.47    &199.95   &215.61 &205.76      &197.87  &203.71\\\hline
\end{tabular}
\end{table}

\subsection{Initialization }
The experiments in this section are conducted under Windows 10 and Matlab R2017a running on a desktop (Intel Core i7-8700, CPU @ 3.20GHz, 16.0G RAM).
We still employ the three color images ``peppafamily'', ``pepper'', and ``colortexture'' to verify the effectiveness of our proposed initialization strategy using the Douglas-Rachford splitting method in Section \ref{sec:initia}.
We test the performances for (1) the alternating projections, (2) DRSM, (3) the alternating projections initialized by DRSM on the best rank-$20$ approximations to these three images.
In the DRSM, we set the parameters $\tau$ and $\alpha$ adaptively.
A large penalty parameter $\tau$ implies a good approximation \label{eq_ind2} to the original problem \label{eq_ind1}.
However, it also leads to a small stepsize, which may cause a slow convergence and also overflows in the floating point arithmetic.
Hence, we gradually increase $\tau$ but keep it constant when it is large enough:
$$
\tau_0 = 1, \quad \tau_{k} =
\left\{
\begin{array}{ll}
  2\tau_{k-1}, & k \leq 1000, \\
  \tau_{k-1}, & k > 1000.
\end{array}
\right.
$$
We adopt the strategy for adaptively choosing the stepsize $\alpha$ in \cite{LP16}:
$$
\alpha_0 = \frac{150}{1+\tau_0}, \quad \alpha_{k} = \max\bigg\{ 0.7\alpha_{k-1},\frac{0.99}{1+\tau_k} \bigg\},
$$
which satisfy the convergence conditions when $k$ is sufficiently large.
\textcolor{red}{
When $\tau_k$ is large, the corresponding $\alpha_k$ will be close to zero.
Then the coefficients in the iteration \eqref{eq_drsm_initial} will be close to either zero or $O(1)$ constants.
Hence, a large parameter $\tau_k$ will not lead to numerical problems when iterating.
}

For the images ``peppafamily'' and ``pepper'', we set the maximum total number of iterations as $5000$ steps, where $500$ steps of DRSM for initialization.
For the image ``colortexture'', for which the involved algorithms converge slower than the other images, we set the maximum total number of iterations as $10000$ steps, where $3000$ steps of DRSM for initialization.
We also cease the iterations when the residual, the Frobenius norms of the real part, is less than $10^{-6}$.

Figure \ref{fig_res} displays the residuals decreasing with the iterations and the running time.
The convergence of the DRSM can be very slow although guaranteed.
One can conclude from the numerical comparisons that the alternating projections converges much faster with the initialization by DRSM than simply taking the original image as the initial point.

We also present the singular values of the original images (quaternion matrices) and the results for different algorithms in Figure \ref{fig_sv}.
The truncations of the singular values with our proposed initialization strategy are the most ``clear'' among these three algorithms.
That is, the final results with the initialization are the closest to an actual rank-$20$ quaternion matrix.

\begin{figure}[!h]
  \centering
  \includegraphics[width=170pt]{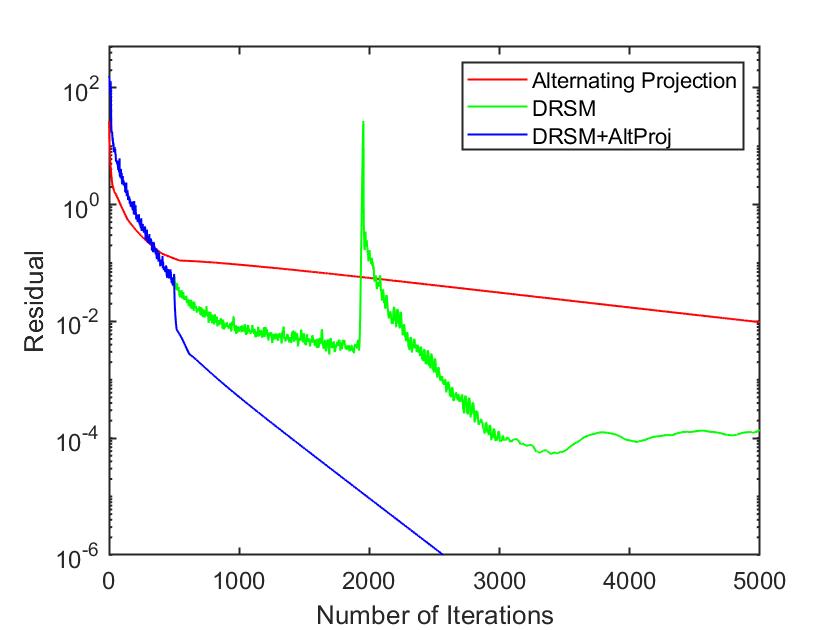}
  \includegraphics[width=170pt]{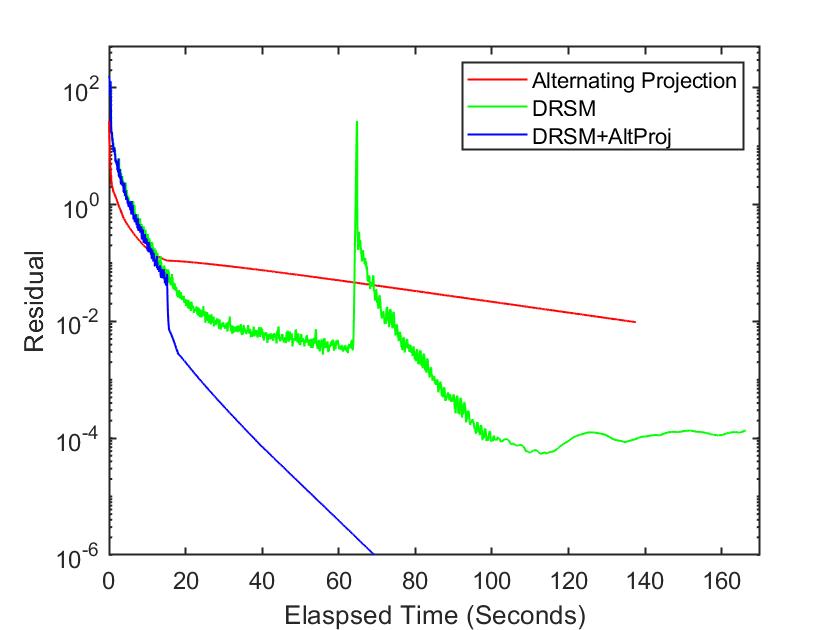}
  \includegraphics[width=170pt]{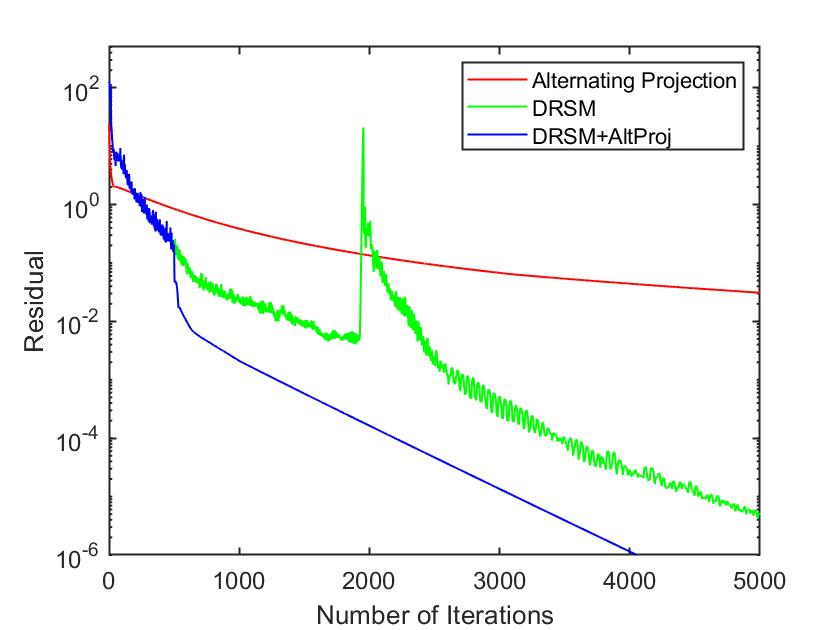}
  \includegraphics[width=170pt]{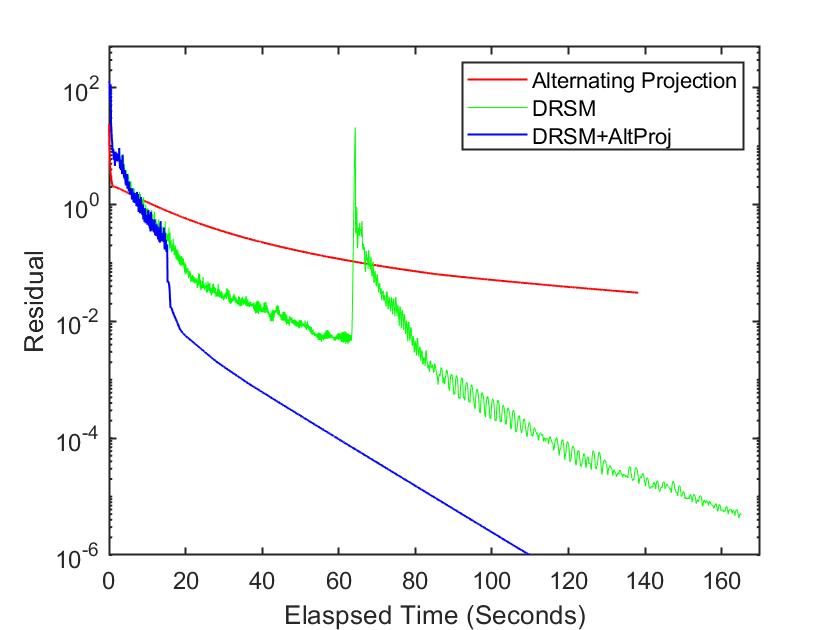}
  \includegraphics[width=170pt]{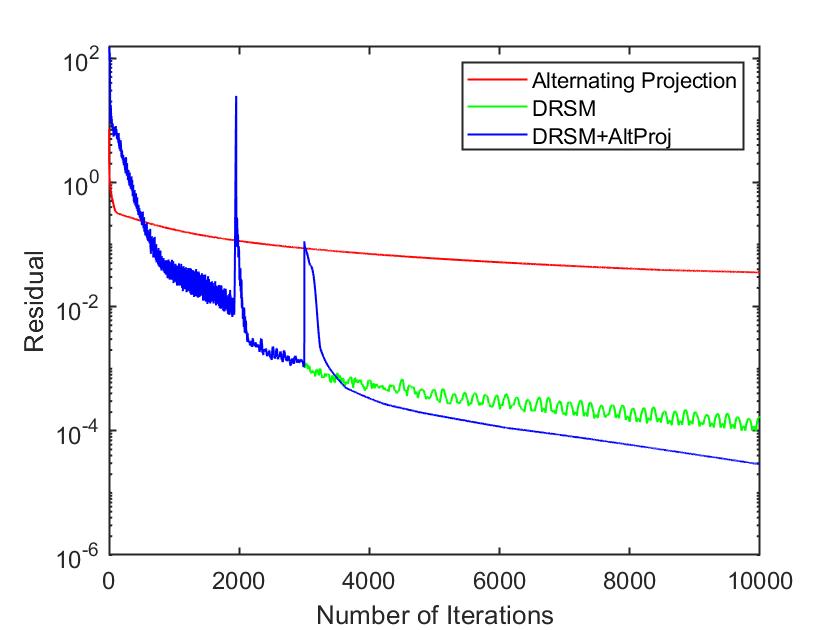}
  \includegraphics[width=170pt]{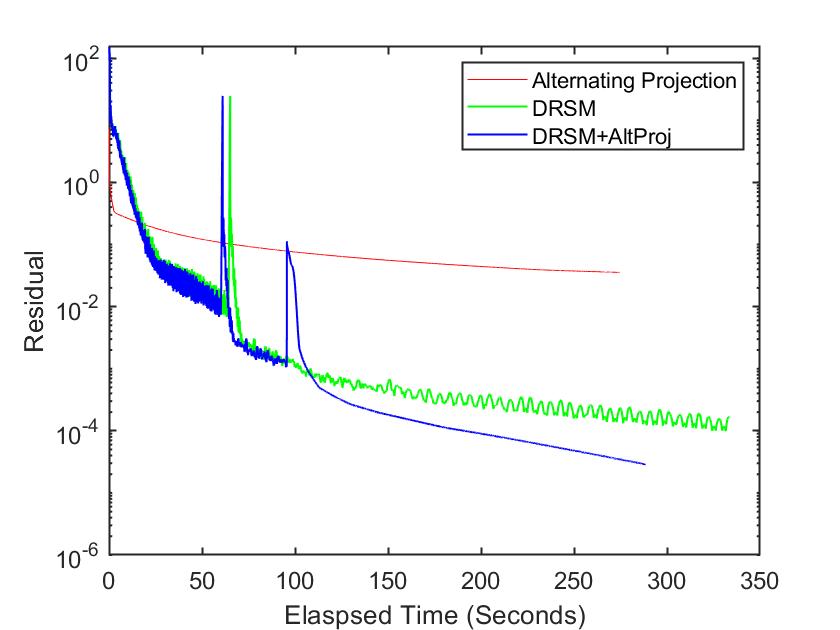}
  \caption{The residuals for the alternating projections, DRSM, the alternating projections initialized by DRSM. These three rows display the cases for ``peppafamily'', ``pepper'', and ``colortexture'', respectively. The figures in the left column show the residuals against the number of iterations, and the figures in the right column show the residuals against the elapsed time.}\label{fig_res}
\end{figure}

\begin{figure}[!h]
  \centering
  \includegraphics[width=170pt]{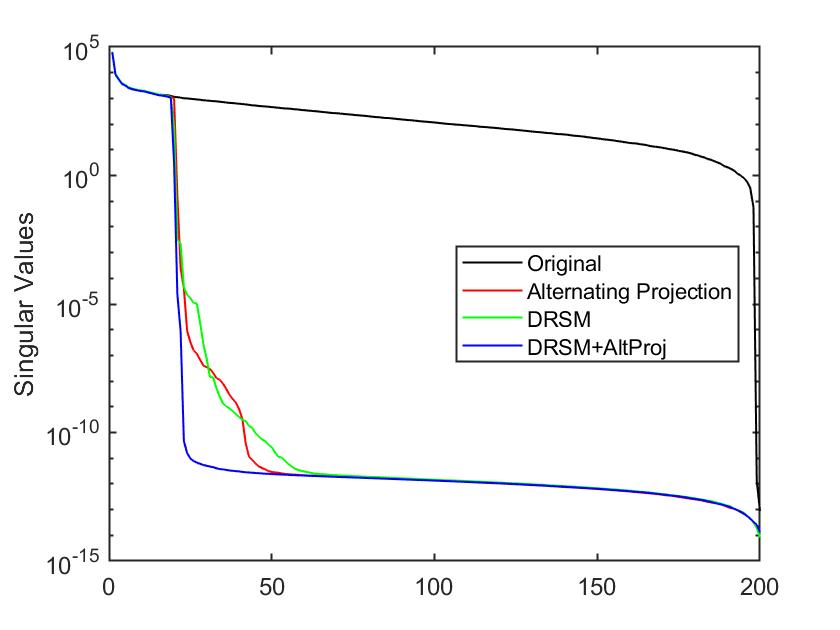}
  \includegraphics[width=170pt]{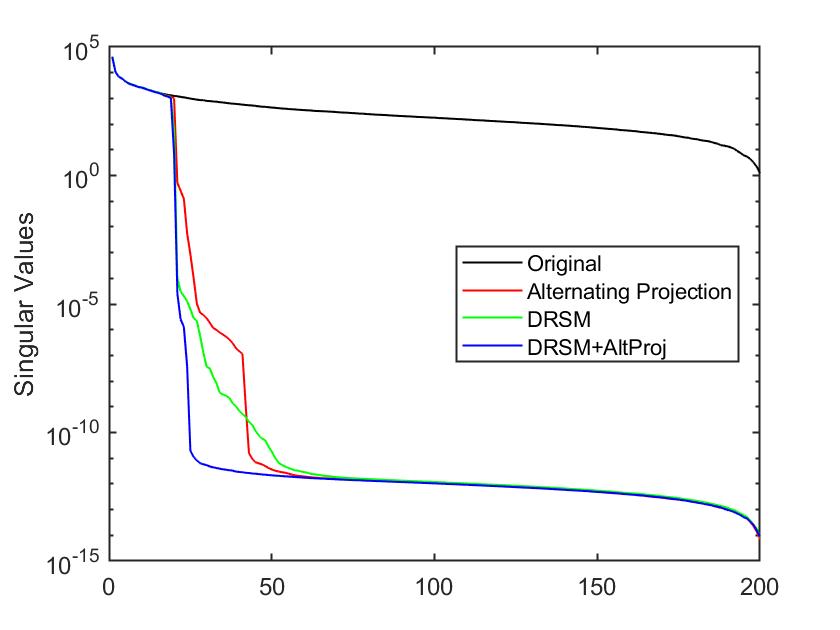}
  \includegraphics[width=170pt]{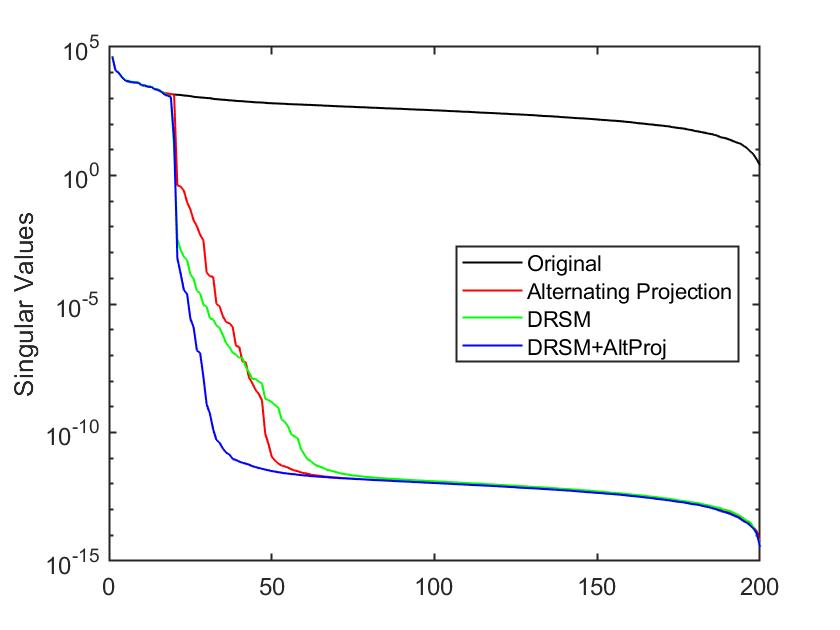}
  \caption{The singular values of the original images and the results for the alternating projections, DRSM, the alternating projections initialized by DRSM. These three figures display the cases for ``peppafamily'', ``pepper'', and ``colortexture'', respectively.}\label{fig_sv}
\end{figure}

\section{Conclusion}\label{sec:con}
We propose the alternating projections method for computing the optimal 
rank-$r$ 
pure quaternion
approximation to any pure quaternion matrix, which consists of the alternating projections onto the rank-$r$ quaternion matrix manifold and the 
pure quaternion matrix manifold.
The linear local convergence for the alternating projections method is proved employing the manifold structures.
In order to guarantee the quality of the limit point and pursue a faster convergence rate, we also propose an initialization strategy using the Douglas-Rachford splitting method to search for an initial point in some neighborhood of the intersection manifold.
Furthermore, we also conduct numerical experiments on both random matrices and real-world color images to illustrate the effect of our proposed alternating projections method and the initialization strategy.


\section*{Appendix}
\vspace{3mm}
\noindent
{\bf Proof of Lemma \ref{qm1}}:  For a rank-$r$ quaternion matrix ${\bf E}=E_{0}+E_{1}{\bf i}+E_{2}{\bf j}+E_{3}{\bf k}\in \mathcal{Q},$   by the elementary transformations it can be expressed as
\begin{align*}
{\bf E}=\left(
    \begin{array}{cc}
      {\bf A} & {\bf B} \\
      {\bf C} & {\bf D} \\
    \end{array}
  \right)
\end{align*}
where ${\bf A}\in \mathbb{H}^{r\times r}$ is invertible, ${\bf B}\in \mathbb{H}^{r\times (m-r)},$ ${\bf C}\in \mathbb{H}^{(n-r)\times r}$ and ${\bf D}\in \mathbb{H}^{(m-r)\times (n-r)}.$  It is easy to find an invertible quaternion matrix
$$
{\bf P}=\left(
          \begin{array}{cc}
            {\bf A}^{-1} & -{\bf A}^{-1}{\bf B} \\
            0 & {\bf I}_{n-k} \\
          \end{array}
        \right)
$$
such that
\begin{align*}
{\bf EP}=\left(
    \begin{array}{cc}
      {\bf A} & {\bf B} \\
      {\bf C} & {\bf D} \\
    \end{array}
  \right)\left(
          \begin{array}{cc}
            {\bf A}^{-1} & -{\bf A}^{-1}{\bf B} \\
            0 &{\bf I}_{n-k} \\
          \end{array}
        \right)=\left(
          \begin{array}{cc}
            {\bf I}_{r} & 0 \\
            {\bf CA^{-1}} & {\bf D-CA^{-1}B} \\
          \end{array}
        \right),
\end{align*}
with  ${\bf D-CA^{-1}B}=0.$ Let
\begin{align*}
{\bf U}=\left\{\left(
    \begin{array}{cc}
      {\bf A} & {\bf B} \\
      {\bf C} & {\bf D} \\
    \end{array}
  \right)\in \mathbb{H}^{m\times n}: {\bf A}~ \text{is ~invertible}
\right\}
\end{align*}
be an open set of $\mathbb{H}^{m\times n}$ which contains ${\bf E}$.
Denote $ \hat{\mathbb{R}}^{4m\times 4n}$ as the set of matrices that possess the structure as the real expression of an $m\times n$ quaternion matrix.
By the projection $\phi$ defined in \eqref{realq}, ${\bf U}$ is  isomorphic to
$$\hat{U}=\left\{\left(
    \begin{array}{cc}
     \hat{ {\bf A}} & \hat{{\bf B}} \\
      \hat{{\bf C}} & \hat{{\bf D}} \\
    \end{array}
  \right)\in \hat{\mathbb{R}}^{4m\times 4n}: {\bf \hat{A}}\in \hat{\mathbb{R}}^{4r\times 4r}~ \text{is ~invertible}
\right\},$$
which is a open subset of $\hat{\mathbb{R}}^{4m\times 4n}.$
 Hence, we can define
$F\circ \phi:{\bf U}\rightarrow  \hat{\mathbb{R}}^{4(m-r)\times 4(n-r)}$ as
\begin{align*}
F\circ \phi\left(
    \begin{array}{cc}
      {{\bf A}} & {{\bf B}} \\
      {{\bf C}} & {{\bf D}} \\
    \end{array}
  \right)=\hat{{\bf D}}-{\bf \hat{C}\hat{A}^{-1}\hat{B}}.
\end{align*}
Clearly, $F\circ \phi$ is smooth. In order to show it is a submersion, we need to show $D(F\circ \phi)({{\bf E}})$ is surjective for each ${{\bf E}}\in {\bf U}.$  Note that $\hat{\mathbb{R}}^{4(m-r)\times 4(n-r)}$ is a vector space, the tangent vectors at $F\circ \phi({{\bf E}})$ can be identified  by the matrices in $\hat{\mathbb{R}}^{4(m-r)\times 4(n-r)}$. Given ${{\bf E}}=\left(
    \begin{array}{cc}
     {{\bf A}} & {{\bf B}} \\
      {{\bf C}} & {{\bf D}} \\
    \end{array}
  \right)
$
and any matrix $X\in \hat{\mathbb{R}}^{4(m-r)\times 4(n-r)},$ define a curve $\tau :(-\xi,\xi)\rightarrow \hat{U}$ by
\begin{align*}
\tau(t)=\left(
    \begin{array}{cc}
      \hat{{\bf A}} & \hat{{\bf B}} \\
      \hat{{\bf C}} & \hat{{\bf D}}+tX \\
    \end{array}
  \right).
\end{align*}
Then
\begin{align*}
(F\circ \phi)_{*}\tau^{'}(0)=(F\circ \phi\circ\tau )^{'}(t)=\frac{d}{dt}|_{t=0}(\hat{{\bf D}}+tX-{\bf \hat{C}\hat{A}^{-1}\hat{B}})=X,
\end{align*}
 where $(F\circ \phi)_{*}$ is the push-forward projection relate $F\circ \phi.$
 Then $F\circ \phi$ is a submersion and so  $\mathcal{Q}\cap \bf{U}$  is an embedded submanifold of $\mathbb{H}^{m\times n}$.
Next, if $\bf{E}'$ is an arbitrary quaternion matrix with $\rank(\bf{E}')=r,$  then it can be transformed to a quaternion matrix in $\bf{U}$ by a rearrangement along its rows and columns.  Let $\omega$ denote such a rearrangement which preserves the quaternion matrix rank. It follows that  ${\bf U}_0 = \omega^{-1}({\bf U})$ is a neighborhood of $\bf{E}'$ and $F\circ \phi\circ \omega: \bf{U}_0\rightarrow \hat{\mathbb{R}}^{4(m-r)\times 4(n-r)}$ is a submersion whose zero level set is $\mathcal{Q}\cap \bf{U}_{0}.$  Thus every point in $\mathcal{Q}$ has a neighborhood ${\bf U}_{0}\subseteq \mathbb{H}^{m\times n}$ such that $\mathcal{Q}\cap {\bf U}_0$ is an embedded submanifold of ${\bf U}_0$, so $\mathcal{Q}$ is an embedded submanifold. Moreover, note that $\dim ((F\circ \phi)_{*}\tau^{'}(0))=4(m+n)r-4r^2$ which is saying that $\mathcal{Q}$ possess dimension $4(m+n)r-4r^2.$ $\quad \Box$

\vspace{3mm}
\noindent
{\bf Proof of Proposition \ref{rankincrease}}:
Denote ${\bf A}^{*{\bf ij}}=A_{0}-A_{1}{\bf i}-A_{2}{\bf j}+A_{3}{\bf k}$ then for two arbitrary quaternion matrices ${\bf B,C}$, we have $({\bf BC})^{*{\bf ij}}={\bf B}^{*{\bf ij}}{\bf C}^{*{\bf ij}}.$ Similarly, we have $({\bf BC})^{*{\bf ik}}={\bf B}^{*{\bf ik}}{\bf C}^{*{\bf ik}}$ and $({\bf BC})^{*{\bf jk}}={\bf B}^{*{\bf jk}}{\bf C}^{*{\bf jk}}.$ Note that $\rank({\bf A})=r$,
 then there exist a full column rank matrix ${\bf U}\in \mathbb{H}^{m\times r}$ and a full row rank matrix ${\bf V}\in \mathbb{H}^{r\times n}$ such that ${\bf A}={\bf U}\cdot {\bf V}$. Moreover,
 \begin{align*}{\bf A}^{*{\bf ij}}={\bf U}^{*{\bf ij}}\cdot {\bf V}^{*{\bf ij}},{\bf A}^{*{\bf ik}}={\bf U}^{*{\bf ik}}\cdot {\bf V}^{*{\bf ik}},\text{~and~}{\bf A}^{*{\bf jk}}={\bf U}^{*{\bf jk}}\cdot {\bf V}^{*{\bf jk}}.
 \end{align*}For the pure quaternion
part of  ${\bf A}$, we have
 \begin{align*}
 {\bf A}_{\text{p}}&=\frac{1}{4}(({\bf A}-{\bf A}^{*{\bf ij}})+({\bf A}-{\bf A}^{*{\bf jk}})+({\bf A}-{\bf A}^{\bf*ik}))=\frac{1}{4}(3{\bf A}-{\bf A}^{\bf*ij}-{\bf A}^{\bf*ik}-{\bf A}^{*{\bf jk}})\\
 &=\frac{1}{4}(3{\bf U}V-{\bf U}^{*{\bf ij}}{\bf V}^{*{\bf ij}}-{\bf U}^{*{\bf ik}}{\bf V}^{*{\bf ik}}-{\bf U}^{*{\bf jk}} {\bf V}^{*{\bf jk}} )\\
 &=\frac{1}{4}(3{\bf U},{\bf U}^{*{\bf ij}},{\bf U}^{*{\bf ik}},{\bf U}^{*{\bf jk}})\left(
                                                                                                         \begin{array}{c}
                                                                                                           {\bf V}   \\
                                                                                                          -{\bf V}^{*{\bf ij}}   \\
                                                                                                          - {\bf V}^{*{\bf ik}}   \\
                                                                                                          - {\bf V}^{*{\bf jk}}   \\
                                                                                                         \end{array}
                                                                                                       \right).
 \end{align*} Then  $r\leq \rank({\bf A}_{\text{p}})\leq 4r.$
This completes the proof. \qed

\end{document}